\documentclass[reqno]{amsart}

\usepackage{amssymb}
\usepackage{amsthm}
\usepackage{amsmath}
\usepackage{mathrsfs}

\newtheorem{theorem}{Theorem}[section]
\newtheorem{corollary}[theorem]{Corollary}
\newtheorem{claim}[theorem]{Claim}
\newtheorem{lemma}[theorem]{Lemma}
\newtheorem{proposition}[theorem]{Proposition}

\theoremstyle{definition}
\newtheorem{definition}[theorem]{Definition}
\newtheorem{remark}[theorem]{Remark}
\newtheorem{example}[theorem]{Example}

\newtheorem{thm}{Theorem}

\newcommand{\C}{\mathbb{C}}
\newcommand{\N}{\mathbb{N}}
\newcommand{\R}{\mathbb{R}}

\newcommand{\mcB}{\mathcal{B}}

\DeclareMathOperator{\Per}{Per}
\DeclareMathOperator{\rank}{rank}
\DeclareMathOperator{\Log}{Log}
\DeclareMathOperator{\var}{var}

\date{\today}

\title[\null]{Compactness of transfer operators and spectral representation of Ruelle zeta functions for super-continuous functions}

\author[\null]{Katsukuni Nakagawa}

\address{K.\ Nakagawa\\
Graduate School of Advanced Science and Engineering\\
Hiroshima University\\
Kagamiyama 1-3-1\\
Higashi-Hiroshima 739-8526\\
Japan}
\email{ktnakagawa@hiroshima-u.ac.jp}

\begin{document}



\begin{abstract}
Transfer operators and Ruelle zeta functions for super-continuous functions on one-sided topological Markov shifts are considered. 
For every super-continuous function, we construct a Banach space on which the associated transfer operator is compact. 
Using this Banach space, we establish the trace formula and spectral representation of Ruelle zeta functions for a certain class of super-continuous functions.
Our results include, as a special case, the classical trace formula and spectral representation for the class of locally constant functions. 
\end{abstract}

\maketitle

\section{Introduction}
\label{sec:introduction}

Let $N\ge2$ be an integer and $\textbf{A}$ an $N\times N$ zero-one matrix. 
We say that $\textbf{A}$ is \emph{aperiodic} if there exists a positive integer $k$ such that all entries of $\textbf{A}^{k}$ are positive. 
In this paper, we always assume that $\textbf{A}$ is aperiodic.
We set
\[
\Sigma_{\textbf{A}}^{+}=\{\omega=(\omega_{m})_{m\in\N\cup\{0\}}\in\{1,\dots,N\}^{\N\cup\{0\}}:\textbf{A}(\omega_{m}\omega_{m+1})=1,\ m\in\N\cup\{0\}\}
\]
and equip $\Sigma_{\textbf{A}}^{+}$ with the product topology. 
Then, $\Sigma_{\textbf{A}}^{+}$ is a compact topological space. 
We define the \emph{shift map} $\sigma_{\textbf{A}}:\Sigma_{\textbf{A}}^{+}\to\Sigma_{\textbf{A}}^{+}$ by 
\[
(\sigma_{\textbf{A}}\omega)_{m}=\omega_{m+1},\quad m\in\N\cup\{0\}.
\]
Then, $\sigma_{\textbf{A}}$ is a continuous mapping. 
We call the dynamical system $(\Sigma_{\textbf{A}}^{+},\sigma_{\textbf{A}})$ a \emph{one-sided topological Markov shift}. 

For $\phi:\Sigma_{\textbf{A}}^+\to\C$ and $m\in\N\cup\{0\}$, we write
\begin{equation}
\label{def:080}
\var_m(\phi)=\sup_{\omega,\omega^{\prime}}|\phi(\omega)-\phi(\omega^{\prime})|,
\end{equation}
where the $\sup_{\omega,\omega^{\prime}}$ is taken over all $\omega,\omega^{\prime}\in\Sigma_{\textbf{A}}^+$ with $\omega_k=\omega^{\prime}_k,\ 0\le k\le m-1$. 
If $\var_m(\phi)^{1/m}\to0$ as $m\to\infty$, then we call $\phi$ a \emph{super-continuous function}. 
(This term is taken from \cite{Quas ETDS 2012}. See also Remark \ref{rem:072} below.)
We set
\[
V=\{\phi:\Sigma_{{\bf A}}^+\to\C:\phi\ \mbox{is a super-continuous function}\}.
\]
For $m\in\N\cup\{0\}$, we also set
\[
L_m=\{\phi\in V:\var_m(\phi)=0\}.
\]
We call an element of $\bigcup_{m\ge0}L_m$ a \emph{locally constant function} on $\Sigma_{{\bf A}}^+$. 
Note that $L_0$ is the set of constant functions and that $L_0\subset L_1\subset\cdots$.
Moreover, for $m\in\N\cup\{0\}$, $L_m$ is a finite-dimensional linear subspace of $V$ and $\dim L_m\le N^m$.
There exists a natural topology of $V$. 
Indeed, for $\theta\in(0,1)$, we define the metric $d_{\theta}$ on $\Sigma_{{\bf A}}^+$ by
\[
d_{\theta}(\omega,\omega^{\prime})=\theta^{m_0},\quad m_0=\min\{m\in\N\cup\{0\}:\omega_{m}\neq\omega^{\prime}_{m}\}
\]
and denote by $F_{\theta}$ the set of complex-valued $d_{\theta}$-Lipschitz continuous functions on $\Sigma_{\textbf{A}}^+$. 
Then, $V=\bigcap_{\theta\in(0,1)}F_{\theta}$ (see Lemma \ref{lem:405} below). 
We denote by $\|\cdot\|_{\theta}$ the Lipschitz norm with respect to the metric $d_{\theta}$.
We equip $V$ with the topology induced by the family of norms $\{\|\cdot\|_{\theta}\}_{\theta\in(0,1)}$. 

Let $f\in V$. 
The \emph{Ruelle transfer operator} $\mathscr{L}_f:V\to V$ of $f$ is defined as follows:
\[
(\mathscr{L}_{f}\phi)(\omega)=\sum_{\omega^{\prime}\in\Sigma_{\textbf{A}}^+:\,\sigma_{\textbf{A}}\omega^{\prime}=\omega}e^{f(\omega^{\prime})}\phi(\omega^{\prime}).
\]
We set
\[
\Lambda_f=\{\lambda\in\C\setminus\{0\}:\lambda\ \mbox{is an eigenvalue of}\ \mathscr{L}_{f}:V\to V\}.
\]
From \cite[Theorem 1]{Pollicott Invent 1986}, $\Lambda_f$ is a discrete subset of $\C\setminus\{0\}$ and each eigenvalue has finite multiplicity. 
Hence, the structure of $\Lambda_f$ is similar to that of the spectrum of a compact operator on a Banach space. 
In fact, if $f$ is locally constant, then there exists a Banach space $\mcB\subset V$ with $\mathscr{L}_f(\mcB)\subset\mcB$ such that $\mathscr{L}_f:\mcB\to\mcB$ is compact and the discrete structure of $\Lambda_f$ comes from the compactness. 
More precisely, the following assertion holds.

\begin{thm}[\text{\cite[Section 3]{Pollicott Invent 1986}}]
\label{thm:060}
\textit{
Let $m\ge2$ and $f\in L_m$. 
Then, $\mathscr{L}_f(L_{m-1})\subset L_{m-1}$ and $\Lambda_f=\sigma(\mathscr{L}_f:L_{m-1}\to L_{m-1})\setminus\{0\}$. Moreover, for $\lambda\in\Lambda_f$, the multiplicity of $\lambda$ as an eigenvalue of $\mathscr{L}_f:V\to V$ coincides with that as an eigenvalue of $\mathscr{L}_f:L_{m-1}\to L_{m-1}$.
}
\end{thm}

Here, for a bounded linear operator $T:E\to E$ on a Banach space $E$, we denote by $\sigma(T:E\to E)$ the spectrum of $T$.

The first aim of this paper is to extend the above theorem to \emph{all} super-continuous $f$. 
Let $\{\theta_m\}_{m\in\N}$ satisfy 
\begin{equation}
\label{cond:227}
\theta_1\ge\theta_2\ge\cdots\ge0,\quad\lim_{m\to\infty}\theta_m=0.
\end{equation}
We set
\begin{equation}
\label{def:787}
\begin{aligned}
\mcB(\{\theta_m\})=\{\phi\in V:\mbox{there}&\ \mbox{exists}\ C\ge0\ \mbox{such that}\\
&\var_m(\phi)\le C\theta_{m+1}^m\ \mbox{for}\ m\in\N\cup\{0\}\}
\end{aligned}
\end{equation}
and write, for $\phi\in\mcB(\{\theta_m\})$,
\begin{equation}
\label{def:788}
\|\phi\|_{\mcB(\{\theta_m\})}=\|\phi\|_{\infty}+\inf\{C\ge0:\var_m(\phi)\le C\theta_{m+1}^m\ \mbox{for}\ m\in\N\cup\{0\}\}.
\end{equation}
It is easy to see that $(\mcB(\{\theta_m\}),\|\cdot\|_{\mcB(\{\theta_m\})})$ is a Banach space.
The first main result of this paper is as follows:

\begin{theorem}
\label{thm:919}
Let $f\in V$ and let $\{\theta_m\}$ satisfy (\ref{cond:227}). 
Assume that
\begin{equation}
\label{cond:037}
\var_m(f)^{1/m}\le\theta_m,\quad m\in\N.
\end{equation}
We write $\mcB=\mcB(\{\theta_m\})$. 
Then the following two assertions hold:
\begin{itemize}
\item[(i)]$\mathscr{L}_f(\mcB)\subset\mcB$ and $\mathscr{L}_f:\mcB\to\mcB$ is compact.
\item[(ii)]$\Lambda_f=\sigma(\mathscr{L}_f:\mcB\to\mcB)\setminus\{0\}$. 
Moreover, for $\lambda\in\Lambda_f$, the multiplicity of $\lambda$ as an eigenvalue of $\mathscr{L}_f:V\to V$ coincides with that as an eigenvalue of $\mathscr{L}_f:\mcB\to\mcB$. 
\end{itemize}
\end{theorem}

We write
\[
\theta_m(f)=\sup_{k\ge m}\var_k(f)^{1/k}
\]
for $f\in V$ and $m\in\N$. 
It is easy to see that the sequence $\theta_m=\theta_m(f),\ m\in\N$ satisfies (\ref{cond:227}) and (\ref{cond:037}).
Moreover, if $m_0\ge2$ and $f\in L_{m_0}\setminus L_{m_0-1}$, then $\mcB(\{\theta_m(f)\})=L_{m_0-1}$.
Thus, Theorem  \ref{thm:919} is an extension of Theorem \ref{thm:060} to all super-continuous $f$.

We are interested in the connection between the spectrum of the transfer operator $\mathscr{L}_f$ and the poles of the Ruelle zeta function $\zeta_{f}(z)$. 
Here, the \emph{Ruelle zeta function} $\zeta_f(z)$ of $f$ is an exponential of a formal power series defined by  
\[
\zeta_{f}(z)=\exp\left(\sum_{q=1}^{\infty}\frac{z^{q}}{q}\sum_{\omega\in\Per_{q}(\sigma_{{\bf A}})}e^{S_{q}f(\omega)}\right),\quad z\in\C,
\]
where, for $q\in\N$, $\Per_{q}(\sigma_{{\bf A}})$ denotes the set of $\omega\in\Sigma_{{\bf A}}^+$ with $\sigma_{{\bf A}}^{q}\omega=\omega$ and $S_{q}f(\omega)=\sum_{k=0}^{q-1}f(\sigma_{\textbf{A}}^{k}\omega)$. 
It is well known that the radius of convergence of the formal power series is not less than $e^{-P(\Re f)}$, where $P(\Re f)$ denotes the \emph{topological pressure} of the real part $\Re f$ of $f$. 
Let 
\[
\lambda_1(f),\lambda_2(f),\dots
\]
be the sequence of non-zero eigenvalues of $\mathscr{L}_f:V\to V$, where each eigenvalue is counted according to its multiplicity and $|\lambda_n(f)|\ge|\lambda_{n+1}(f)|$ holds for $n\in\N$. 
(If the number of the eigenvalues is finite, say, $M$, then we put $\lambda_n(f)=0$ for $n>M$.) 
The following theorem is an immediate consequence of \cite[Corollary 6]{Haydn ETDS 1990}.

\begin{thm}
\label{thm:415}
\textit{
Let $f\in V$. 
Then, $\zeta_{f}(z)^{-1}$ admits a holomorphic extension to $\C$ and its zeros are exactly $\{\lambda_n(f)^{-1}:n\in\N,\ \lambda_n(f)\neq0\}$. 
Moreover, the order of each zero coincides with the multiplicity of the corresponding eigenvalue.
}
\end{thm}

On the other hand, if $f$ is locally constant, then we have the next Weierstrass canonical product form of $\zeta_f$.

\begin{thm}[\text{\cite[Section 3]{Pollicott Invent 1986}}]
\label{thm:525}
\textit{
Let $f\in V$ be locally constant. 
Then, $\{n\in\N:\lambda_n(f)\neq0\}$ is a finite set and
\begin{equation}
\zeta_f(z)^{-1}=\prod_{n=1}^{\infty}(1-z\lambda_n(f)).
\label{eq:782}
\end{equation}
}
\end{thm}

Equation (\ref{eq:782}) means that the entire function $\prod_{n=1}^{\infty}(1-z\lambda_n(f))$ is a holomorphic extension of $\zeta_f(z)^{-1}$ to $\C$. 
Thus, for a locally constant $f$, we also obtain an analog of Theorem \ref{thm:415} by (\ref{eq:782}). 
We call (\ref{eq:782}) the \emph{spectral representation} of $\zeta_f(z)$. 

The second aim of this paper is to establish the representation (\ref{eq:782}) for a wider class of $f\in V$.
To this end, we consider the following condition for $f\in V$ and $r\in(0,1)$:
\begin{equation}
\label{cond:209}
\var_m(f)^{1/m}=O(r^m)\quad\mbox{as}\quad m\to\infty,
\end{equation}
that is, $\limsup_{m\to\infty}\var_m(f)^{1/m}/r^m<\infty$. 
If $f$ is locally constant, then (\ref{cond:209}) is valid for any $r\in(0,1)$. (For each $r\in(0,1)$, we give an example of non-locally constant $f\in V$ satisfying (\ref{cond:209}) in Example \ref{exa:004} below.)

Here is the second maim result of this paper. 

\begin{theorem}
\label{thm:925}
Let $f\in V$ and let $r\in(0,1)$ satisfy (\ref{cond:209}). 
Then, for $p>0$ with 
\begin{equation}
\label{cond:468}
r^{2p}e^{h_{\mathrm{top}}(\sigma_{{\bf A}})}<1,
\end{equation}
the following three assertions hold:
\begin{itemize}
\item[(i)]$\sum_{n=1}^{\infty}|\lambda_n(f)|^p<\infty$.
\item[(ii)]For $q\in\N$ with $q\ge p$, we have
\begin{equation}
\label{eq:008}
\sum_{n=1}^{\infty}\lambda_n(f)^q=\sum_{\omega\in\Per_q(\sigma_{{\bf A}})}e^{S_qf(\omega)}.
\end{equation}
\item[(iii)]Let $k_0$ be the smallest $k\in\N\cup\{0\}$ such that $\sum_{n=1}^{\infty}|\lambda_n(f)|^{k+1}<\infty$ and (\ref{eq:008}) holds for $q\in\N$ with $q>k$.
We set $E(z,k_0)=(1-z)\exp(\sum_{k=1}^{k_0}z^k/k),z\in\C$.
Then, the infinite product $\prod_{n=1}^{\infty}E(z\lambda_n(f),k_0)$ converges uniformly on any compact set of $\C$ and we have
\begin{equation}
\label{eq:962}
\zeta_f(z)^{-1}=\exp\left(-\sum_{q=1}^{k_0}\frac{z^q}{q}\sum_{\omega\in\Per_q(\sigma_{{\bf A}})}e^{S_qf(\omega)}\right)\prod_{n=1}^{\infty}E(z\lambda_n(f),k_0).
\end{equation}
\end{itemize}
\end{theorem}

Note that if $p\le1$, then $k_0=0$, and hence, (\ref{eq:962}) yields the spectral representation (\ref{eq:782}) of $\zeta_f(z)$. 

Equations like (\ref{eq:008}), which give a connection between the poles of a zeta function and the spectrum of the associated transfer operator, are often called \emph{trace formulas}. 
The trace formulas for dynamical zeta functions are widely studied in differentiable dynamical systems; see, e.g., \cite{Fried 1986, Jezequel 2018, Ruelle Invent 76, Ruelle IHES 90}. 

To prove Theorem \ref{thm:925} above, we introduce the operator ideal $\mathfrak{L}_p^{(a)}(E)$. 
Let $E$ be a Banach space.
We denote by $\mathfrak{L}(E)$ the set of bounded linear operators on $E$. 
For $p>0$, we set
\[
\mathfrak{L}_p^{(a)}(E)=\left\{T\in\mathfrak{L}(E):\sum_{n=1}^{\infty}a_n(T)^p<\infty\right\},
\]
where, for $n\in\N$, $a_n(T)$ denotes the $n$-th \emph{approximation number} of $T$ defined by
\begin{equation}
\label{def:785}
a_n(T)=\inf\{\|T-A\|:A\in\mathfrak{L}(E),\rank A<n\}.
\end{equation}
It is easy to see that any element of $\mathfrak{L}_p^{(a)}(E)$ is a compact operator and $\mathfrak{L}_p^{(a)}(E)$ is a left and right ideal of $\mathfrak{L}(E)$, that is, $\mathfrak{L}_p^{(a)}(E)$ is closed under addition and scalar multiplication, and $AT,TA\in\mathfrak{L}_p^{(a)}(E)$ for $A\in\mathfrak{L}(E)$ and $T\in\mathfrak{L}_p^{(a)}(E)$.
Moreover, for $T\in\mathfrak{L}_p^{(a)}(E)$, the following two assertions hold (see \cite[Theorems 3.6.3 and 4.2.26]{Pietsch} for the proof):
\begin{itemize}
\item[(I)] Let $\lambda_1(T),\ \lambda_2(T),\dots$ be the sequence of non-zero eigenvalues of $T$, where each eigenvalue is counted according to its multiplicity and $|\lambda_n(T)|\ge|\lambda_{n+1}(T)|$ holds for $n\in\N$. 
(If the number of the eigenvalues is finite, say, $M$, then we put $\lambda_n(T)=0$ for $n>M$.) 
Then, $\sum_{n=1}^{\infty}|\lambda_n(T)|^p<\infty$.
\item[(II)] Let $p=1$ and we write
\[
|||\,T\,|||=\sum_{n=1}^{\infty}a_n(T),\quad\tau(T)=\sum_{n=1}^{\infty}\lambda_n(T).
\]
Let $T_1,T_2,\ldots\in\mathfrak{L}_1^{(a)}(E)$. 
If $\lim_{m\to\infty}|||\,T_m-T\,|||=0$, then $\lim_{m\to\infty}\tau(T_m)=\tau(T)$.
\end{itemize}
Applying the theory of $\mathfrak{L}_p^{(a)}(E)$ to $E=\mcB$, we prove Theorem \ref{thm:925} in Section \ref{sec:main02}.

This paper is organized as follows. 
In Section \ref{sec:approximation}, we give preliminary definitions and basic facts. 
In Section \ref{sec:estimates}, we prove some estimates for the proofs of the main results, i.e., Theorems \ref{thm:919} and \ref{thm:925}.
In Section \ref{sec:main01}, we prove Theorem \ref{thm:919},
and in Section \ref{sec:main02}, we prove Theorem \ref{thm:925}.
In Appendix \ref{appendix:transfer operator}, we study the properties of transfer operators acting on $V$. 
It is natural to hope that the transfer operator $\mathscr{L}_f:V\to V$ is a compact operator. 
However, in Appendix \ref{appendix:transfer operator},  we prove that this is not the case for any $f\in V$.
Moreover, we give an example of $f\in V$ such that $\sum_{n=1}^{\infty}|\lambda_n(f)|=\infty$.
In Appendix \ref{appendix:TVS}, we study the properties of $V$ itself. 
We will see that $V$ is naturally a nuclear space. 
Moreover, we prove that $V$ has many non-trivial (i.e., non-locally constant) elements. 
More precisely, we prove that the set of non-locally constant elements of $V$ is a residual subset of $V$. 
In Appendix \ref{appendix:eigenvalues}, we study the asymptotic behavior of eigenvalues of transfer operators.
Using the Weyl inequality in Banach spaces (see, e.g., \cite[Theorem 2.a.6]{Konig}), we obtain an asymptotic behavior of $\{\lambda_n(f)\}_{n\in\N}$ for $f\in V$ satisfying (\ref{cond:209}) for some $r\in(0,1)$.
In Appendix \ref{appendix:eigenvalues}, we give an asymptotic behavior of $\{\lambda_n(f)\}_{n\in\N}$ for arbitrary $f\in V$, using a recent result of Demuth et al. \cite{Demuth et al JFA 2015}.


\section{Preliminaries}
\label{sec:approximation}

An element of $\bigcup_{m\in\N\cup\{0\}}\{1,\dots,N\}^m$ is called a \emph{word}. 
For $m\in\N\cup\{0\}$ and a word $w\in\{1,\dots,N\}^m$, we write $|w|= m$. 
Moreover, we write $w=w_0\cdots w_{|w|-1}$ for a word $w$, where $w_k\in\{1,...,N\},\ 0\le k\le|w|-1$. 
The \emph{empty word} is the unique word $w$ with $|w|=0$. 
A word $w$ with $|w|\ge2$ is \emph{self-avoiding} if $w_k\neq w_l$ for $0\le k\neq l\le|w|-1$.
We define the new word $vw$ for two words $v,w$ by $vw=v_0\cdots v_{|v|-1}w_{0}\cdots w_{|w|-1}$.
Moreover, for a word $w$ with $|w|\ge1$, we define $w^*\in\{1,\dots,N\}^{\N\cup\{0\}}$ by $w^*=www\cdots$. 
A word $w$ is said to be $\bf A$-\emph{admissible} if $|w|\ge2$ and $A(w_{k}w_{k+1})=1$ for $0\le k\le|w|-1$. 
For $\omega\in\Sigma_{{\bf A}}^{+}$ and $m\in\N\cup\{0\}$, we define the word $\omega|m\in\{1,\dots,N\}^m$ by
\[
\omega|m=\begin{cases}
\omega_{0}\cdots\omega_{m-1}&(m\ge1),\\
\mbox{the empty word}&(m=0).
\end{cases}
\]
For $m\in\N\cup\{0\}$ and $w\in\{1,\dots,N\}^m$, we set
\[
[w]=\{\omega\in\Sigma_{{\bf A}}^{+}:\omega|m=w\}.
\]
A point $\omega\in\Sigma_{{\bf A}}^{+}$ is said to be \emph{periodic} if $\sigma_{{\bf A}}^{q}\omega=\omega$ for some $q\in\N$. 
For a periodic point $\omega$, its \emph{period} is the smallest $q\in\N$ such that $\sigma_{{\bf A}}^{q}\omega=\omega$.  
We denote by $\Per_{q}(\sigma_{{\bf A}})$ the set of periodic points $\omega\in\Sigma_{{\bf A}}^+$ with $\sigma_{{\bf A}}^{q}\omega=\omega$.

We recall that the $N\times N$ zero-one matrix ${\bf A}$ is assumed to be aperiodic, that is, all entries of ${{\bf A}}^{k}$ are positive for some positive integer $k$. 
The following lemma is needed in Appendices \ref{appendix:transfer operator} and \ref{appendix:TVS}.

\begin{lemma}
\label{lem:203}
At least one row of ${\bf A}$ has more than two entries which are equal to one, and similarly for columns.
\end{lemma}
\begin{proof}
Assume that every row has just one entry which is equal to 1. 
Then there exists a permutation $\tau$ of the set $\{1,...,N\}$ such that ${\bf A}(ij)=1\ (j=\tau(i)),\ =0\ (j\neq\tau(i))$. 
Thus, ${\bf A}^{k}(ij)=1\ (j=\tau^{k}(i)),\ =0\ (j\neq\tau^{k}(i))$ for $k\ge1$,  and hence, ${\bf A}$ is not aperiodic.
The transpose ${\bf A}^{\mathrm{T}}$ of ${\bf A}$ is also aperiodic.
Hence, the assertion for columns also holds.
\end{proof}

We recall from Section \ref{sec:introduction} that, for $\theta\in(0,1)$, $F_{\theta}$ denotes the set of complex-valued functions on $\Sigma_{{\bf A}}^+$ that are Lipschitz continuous with respect to the metric $d_{\theta}$ on $\Sigma_{{\bf A}}^+$. 
The Lipschitz norm $\|\phi\|_{\theta}$ for $\phi\in F_{\theta}$ is defined by $\|\phi\|_{\theta}=\|\phi\|_{\infty}+[\phi]_{\theta}$, where
\[
\|\phi\|_{\infty}=\max_{\omega\in\Sigma_{{\bf A}}^+}|\phi(\omega)|,\qquad[\phi]_{\theta}=\sup_{\omega,\omega^{\prime}\in\Sigma_{{\bf A}}^+:\ \omega\neq\omega^{\prime}}\frac{|\phi(\omega)-\phi(\omega^{\prime})|}{d_{\theta}(\omega,\omega^{\prime})}.
\]
Then, $(F_{\theta},\|\cdot\|_{\theta})$ is a Banach space. 
Moreover, we easily see that if $\theta<\theta^{\prime}$, then $\|\phi\|_{\theta^{\prime}}\le\|\phi\|_{\theta}$ for $\phi\in F_{\theta}$, and hence, $F_{\theta}\subset F_{\theta^{\prime}}$. 

Let us recall the following definition of a super-continuous function.
\begin{definition}
\label{def:396}
A \emph{super-continuous function} on $\Sigma_{{\bf A}}^+$ is a function $\phi:\Sigma_{{\bf A}}^+\to\C$ such that $\var_{m}(\phi)^{1/m}\to0$ as $m\to\infty$.
\end{definition}
We denote by $V$ the set of all super-continuous functions on $\Sigma_{{\bf A}}^+$.

\begin{remark}
\label{rem:072}
A super-continuous function on a topological Markov shift was first defined by Quas and Siefken in \cite{Quas ETDS 2012} as follows: $\phi:\Sigma_{{\bf A}}^+\to\C$ is called a super-continuous function if there exists a positive and non-increasing sequence $\{A_m\}_{m\in\N}$ such that $\var_m(\phi)\le A_m$ for $m\in\N$ and $A_{m+1}/A_m\to0$ as $m\to\infty$.
Let $V^{\prime}$ be the set of super-continuous functions in the sense of \cite{Quas ETDS 2012}. 
Then, $V=V^{\prime}$.
Indeed, $V^{\prime}\subset V$ is obvious.
Let $\phi\in V$.
We may assume that $0<\var_m(\phi)<1$ for any $m\in\N$. 
We set $A_m=\theta_m^m$ for $m\in\N$, where $\theta_m=\inf\{\theta\in(0,1):\var_k(\phi)\le\theta^k,k\ge m\}$.
Then, $\{A_m\}_{m\in\N}$ is positive and non-increasing. 
Moreover, $\var_m(\phi)\le A_m$ for $m\in\N$ and $A_{m+1}/A_m\to0$ as $m\to\infty$. 
Thus, $\phi\in V^{\prime}$.
\end{remark}

\begin{lemma}
\label{lem:405}
We have $V=\bigcap_{\theta\in(0,1)}F_{\theta}$.
\end{lemma}
\begin{proof}
Let $\phi\in V$. 
Fix $\theta\in(0,1)$. 
For sufficiently large $m$, we have $\var_{m}(\phi)^{1/m}\le\theta$, and hence, we have $\var_{m}(\phi)\le\theta^m$. 
This implies $\phi\in F_{\theta}$.

Let $\phi\in\bigcap_{\theta\in(0,1)}F_{\theta}$. 
For $\theta\in(0,1)$, there exists $C>0$ such that $\var_{m}(\phi)\le C\theta^m$ for $m\in\N\cup\{0\}$, and hence, $\limsup_{n\to\infty}\var_{m}(\phi)^{1/m}\le\theta$. 
Letting $\theta\to0$, we obtain $\lim_{n\to\infty}\var_{m}(\phi)^{1/m}=0$.
\end{proof}

Recall from Section \ref{sec:introduction} that $V$ is equipped with the topology induced by the family of norms $\{\|\cdot\|_{\theta}\}_{\theta\in(0,1)}$. 
By the definition of the topology of $V$, we easily see that $\mathscr{L}_f:V\to V$ is continuous for $f\in V$.
Moreover, since $\|\cdot\|_{\theta^{\prime}}\le\|\cdot\|_{\theta}$ for $\theta<\theta^{\prime}$, we see that the topology of $V$ coincides with that induced by the countable subfamily $\{\|\cdot\|_{1/(m+1)}\}_{m\in\N}$. 
Hence, we obtain the following proposition:

\begin{proposition}
\label{Prop:715}
$V$ is a Fr\'{e}chet space. 
\end{proposition}

Fix a Borel probability measure $\mu$ on $\Sigma_{{\bf A}}^{+}$ such that $\mu(G)>0$ for every non-empty open set $G$ of $\Sigma_{{\bf A}}^+$. 
(The Gibbs measure for a real-valued function in $F_{\theta}$ satisfies this condition; see, e.g., \cite[Chapter 3]{Parry-Pollicott}.)
Let $C(\Sigma_{{\bf A}}^+)$ be the set of complex-valued continuous functions on $\Sigma_{{\bf A}}^+$.
For $m\in\N$, we define a finite-rank operator $E_{m}:C(\Sigma_{{\bf A}}^+)\to L_{m}$ by
\begin{equation}
\label{def:228}
(E_{m}\phi)(\omega)=\frac{1}{\mu([\omega|m])}\int_{[\omega|m]}\phi\,d\mu.
\end{equation}
Notice that $E_{m}\phi=\phi$ for $\phi\in L_m$. 
In addition to (\ref{def:080}), we write 
\[
V_{m}^{\theta}(\phi)=\sup_{k\ge m}\frac{\var_{k}(\phi)}{\theta^k}
\]
for $\phi:\Sigma_{{\bf A}}^+\to\C,m\in\N\cup\{0\}$ and $\theta\in(0,1)$. 
Notice that $[\phi]_{\theta}=V_{0}^{\theta}(\phi)$ for $\phi\in F_{\theta}$. 
The next lemma will be used in Sections \ref{sec:estimates}--\ref{sec:main02} and Appendix \ref{appendix:TVS}. 

\begin{lemma}
\label{lem:681}
Let $\phi\in C(\Sigma_{{\bf A}}^+),m\in\N$ and $\theta,\theta^{\prime}\in(0,1)$.
\begin{itemize}
\item[(i)]If $\phi$ is real-valued, then so is $E_m\phi$ and $\max E_m\phi\le\max\phi$.
\item[(ii)]For $k\in\N\cup\{0\}$, $\var_k(E_m\phi)\le\var_k(\phi)$.
\item[(iii)]If $\phi\in F_{\theta}$, then $\|\phi-E_{m}\phi\|_{\infty}\le V_{m}^{\theta}(\phi)\theta^{m}$.
\item[(iv)]If $\phi\in F_{\theta}$ and $\theta<\theta^{\prime}$, then $\|\phi-E_{m}\phi\|_{\theta^{\prime}}\le3V_{m}^{\theta}(\phi)\left(\theta/\theta^{\prime}\right)^m$.
\end{itemize}
\end{lemma}
\begin{proof}
(i) is obvious. 

(ii)\ Let $\omega,\omega^{\prime}\in\Sigma_{{\bf A}}^+$ satisfy $\omega|k=\omega^{\prime}|k$. 
We show that $|(E_m\phi)(\omega)-(E_m\phi)(\omega^{\prime})|\le\var_k(\phi)$.
We first assume $k\ge m$.
Then, $(E_m\phi)(\omega)=(E_m\phi)(\omega^{\prime})$.
We next assume $k<m$. 
Then,
\begin{align*}
|(E_m\phi)&(\omega)-(E_m\phi)(\omega^{\prime})|\\
\le&\frac{1}{\mu([\omega|m])\,\mu([\omega^{\prime}|m])}\int_{[\omega|m]}\left(\int_{[\omega^{\prime}|m]}|\phi(\xi)-\phi(\xi^{\prime})|\,\mu(d\xi^{\prime})\right)\mu(d\xi).
\end{align*}
Let $\xi\in[\omega|m]$ and $\xi^{\prime}\in[\omega^{\prime}|m]$. 
Then, $\xi|k=\omega|k=\omega^{\prime}|k=\xi^{\prime}|k$ since $k<m$, and hence, $|\phi(\xi)-\phi(\xi^{\prime})|\le\var_k(\phi)$. 
Thus, $|(E_m\phi)(\omega)-(E_m\phi)(\omega^{\prime})|\le\var_k(\phi)$.

(iii)\ Let $\omega\in\Sigma_{{\bf A}}^+$. 
We have
\[
|\phi(\omega)-(E_m\phi)(\omega)|\le\frac{1}{\mu([\omega|m])}\int_{[\omega|m]}|\phi(\omega)-\phi(\omega^{\prime})|\,\mu(d\omega^{\prime}).
\]
If $\omega^{\prime}\in[\omega|m]$, then $|\phi(\omega)-\phi(\omega^{\prime})|\le\var_m(\phi)\le V_m^{\theta}(\phi)\theta^m$. 
Thus, (iii) follows.

(iv)\ By (iii), we have $\|\phi-E_m\phi\|_{\infty}\le V_m^{\theta}(\phi)(\theta/\theta^{\prime})^m$.
Therefore, it is enough to show that
\begin{equation}
\label{ineq:206}
\frac{\var_k(\phi-E_m\phi)}{(\theta^{\prime})^k}\le2V_m^{\theta}(\phi)\left(\frac{\theta}{\theta^{\prime}}\right)^m,\quad k\in\N\cup\{0\}.
\end{equation}
First we assume $k\ge m$. 
Then, 
\[
\var_k(\phi-E_m\phi)=\var_k(\phi)\le V_k^{\theta}(\phi)\theta^k\le V_m^{\theta}(\phi)(\theta/\theta^{\prime})^m(\theta^{\prime})^k,
\]
and hence, $\var_k(\phi-E_m\phi)/(\theta^{\prime})^k\le V_m^{\theta}(\phi)(\theta/\theta^{\prime})^m$.
Next we assume $k<m$. 
From (iii), $\var_k(\phi-E_m\phi)\le2\|\phi-E_m\phi\|_{\infty}\le2V_m^{\theta}(\phi)\theta^m$,
and hence, $\var_k(\phi-E_m\phi)/(\theta^{\prime})^k\le2V_m^{\theta}(\phi)(\theta/\theta^{\prime})^m(\theta^{\prime})^{m-k}\le2V_m^{\theta}(\phi)(\theta/\theta^{\prime})^m$. Combining, we obtain (\ref{ineq:206}).
\end{proof}

We note the next easy inequality:
\begin{equation}
\label{ineq:256}
|e^z-e^{z^{\prime}}|\le3e^{\max(\Re z,\Re z^{\prime})}|z-z^{\prime}|,\quad z,z^{\prime}\in\C.
\end{equation}
The following Lasota-Yorke type inequality is well known and a key tool for the proofs of the main results. 

\begin{lemma}
\label{lem:302}
For $f,\phi\in C(\Sigma_{{\bf A}}^+)$ and $k\in\N$, we have
\[
\var_{k}(\mathscr{L}_{f}\phi)\le Ne^{\max\Re f}\{3\var_{k+1}(f)\|\phi\|_{\infty}+\var_{k+1}(\phi)\}.
\]
\end{lemma}
\begin{proof}
Let $\omega,\,\omega^{\prime}\in\Sigma_{{\bf A}}^+$ satisfy $\omega|k=\omega^{\prime}|k$.
Since $\omega_{0}=\omega^{\prime}_{0}$, we have $\{i:{\bf A}(i\omega_0)=1\}=\{i:{\bf A}(i\omega_0^{\prime})=1\}$. 
Therefore, 
\begin{align*}
|(\mathscr{L}_{f}\phi)(&\omega)-(\mathscr{L}_{f}\phi)(\omega^{\prime})|\\
\le&\sum_{i:{\bf A}(i\omega_0)=1}\{|e^{f(i\omega)}||\phi(i\omega)-\phi(i\omega^{\prime})|+|e^{f(i\omega)}-e^{f(i\omega^{\prime})}||\phi(i\omega^{\prime})|\}.
\end{align*}
Let ${\bf A}(i\omega_0)=1$. 
We easily have $|e^{f(i\omega)}||\phi(i\omega)-\phi(i\omega^{\prime})|\le e^{\max\Re f}\var_{k+1}(\phi)$. 
Moreover, by (\ref{ineq:256}), we have $|e^{f(i\omega)}-e^{f(i\omega^{\prime})}||\phi(i\omega)|\le3e^{\max\Re f}\var_{k+1}(f)\|\phi\|_{\infty}$. 
Thus, the desired inequality holds.
\end{proof}

\section{Some estimates for the proofs of the main results}
\label{sec:estimates}

Fix a sequence $\{\theta_m\}_{m\in\N}$ satisfying (\ref{cond:227}).
Recall from (\ref{def:787}) and (\ref{def:788}) the definitions of the space $\mcB(\{\theta_m\})$ and the norm $\|\cdot\|_{\mcB(\{\theta_m\})}$, respectively.
In this section, we write $\mcB=\mcB(\{\theta_m\})$ and $\|\cdot\|=\|\cdot\|_{\mcB(\{\theta_m\})}$ for the sake of simplicity.
We set
\begin{equation}
\label{eq:224}
C=\sup_{m\in\N}\theta_m.
\end{equation}

We begin with the following easy lemma (we omit the proof):

\begin{lemma}
\label{lem:293}
For $m\in\N$, the following three assertions hold:
\begin{itemize}
\item[(i)]If $\theta_m=0$, then $\mcB\subset L_{m-1}$.
\item[(ii)]If $\theta_m>0$, then $L_{m-1}\subset\mcB$. 
\item[(iii)]If $m\ge2,\theta_m=0$ and $\theta_{m-1}>0$, then $\mcB=L_{m-1}$.
\end{itemize}
\end{lemma}

Let $E_m$ be as in (\ref{def:228}).

\begin{corollary}
\label{cor:488}
For $m\in\N$, we have $E_m(\mcB)\subset\mcB$.
\end{corollary}
\begin{proof}
First, we consider the case in which $\theta_1=0$. 
Then, $\mcB=L_0$ from Lemma \ref{lem:293} (i), and hence, $E_m(L_0)=L_0$. 

Next, we consider the case in which $\theta_1>0$ and $\theta_{m}=0$ for some $m\ge2$. 
Take $m_0\ge2$ so that $\theta_{m_0}=0$ and $\theta_{m_0-1}>0$.
From Lemma \ref{lem:293} (iii), $\mcB=L_{m_0-1}$. 
Thus, if $m\ge m_0-1$, then $E_m(L_{m_0-1})=L_{m_0-1}$, and if $m<m_0-1$, then $E_m(L_{m_0-1})\subset L_{m}\subset L_{m_0-1}$.

Finally, we consider the case in which $\theta_m>0$ for all $m\ge2$. 
From Lemma \ref{lem:293} (ii), $\bigcup_{m\ge0}L_m\subset\mcB$. Thus, $E_m(\mcB)\subset L_m\subset\mcB$.
\end{proof}

\begin{lemma}
\label{lem:046}
Let $m\in\N,k\in\N\cup\{0\}$ and $\phi\in\mcB$.
\begin{itemize}
\item[(i)]$\|\phi-E_m\phi\|_{\infty}\le\|\phi\|\theta_{m+1}^m$.
\item[(ii)]$\var_k(\phi-E_m\phi)\le2\|\phi\|\theta_{m+1}^m$. 
Moreover, if $k\ge m$, then $\var_k(\phi-E_m\phi)\le\|\phi\|\theta_{k+1}^k$.
\item[(iii)]If $\theta_{m+1}\le1$, then $\|I-E_m\|_{\mcB\to\mcB}\le3$.
\end{itemize}
\end{lemma}
\begin{proof}
(i)\ If $\omega,\omega^{\prime}\in\Sigma_{{\bf A}}^+$ satisfy $\omega|m=\omega^{\prime}|m$, then $|\phi(\omega)-\phi(\omega^{\prime})|\le\var_m(\phi)\le\|\phi\|\theta_{m+1}^m$.
Thus, (i) follows from the same argument as that in the proof of Lemma \ref{lem:681} (iii).  

(ii)\ From (i), $\var_k(\phi-E_m\phi)\le2\|\phi-E_m\phi\|_{\infty}\le2\|\phi\|\theta_{m+1}^m$.
If $k\ge m$, then $\var_k(\phi-E_m\phi)=\var_k(\phi)\le\|\phi\|\theta_{k+1}^k$. 

(iii)\ Take $\phi\in\mcB$ so that $\|\phi\|\le1$. 
By (i) and $\theta_{m+1}\le1$, we have $\|\phi-E_m\phi\|_{\infty}\le1$.
Hence, it is enough to show that $\var_k(\phi-E_m\phi)\le2\theta_{k+1}^k$ 
for $k\in\N\cup\{0\}$.
If $k\ge m$, then, from the latter part of (ii), $\var_k(\phi-E_m\phi)\le\theta_{k+1}^k$.
If $k<m$, then, from the former part of (ii) and $\theta_{m+1}\le1$, $\var_k(\phi-E_m\phi)\le2\theta_{m+1}^m\le2\theta_{k+1}^k$.
\end{proof}

Take $b_1,b_2>0$. 
We consider the following condition for $g\in V$:
\begin{equation}
\label{condi:029}
e^{\max\Re g}\le b_1\quad\mbox{and}\quad\var_{k}(g)\le b_2\theta_k^k\quad\mbox{for}\quad k\in\N\cup\{0\}.
\end{equation}
\begin{lemma}
\label{lem:035}
There exists $C_1>0$, depending only on $b_1$ and $b_2$, such that the following inequality holds for $k\in\N\cup\{0\},\phi\in\mcB$ and $g\in V$ satisfying (\ref{condi:029}):
\[
\var_{k}(\mathscr{L}_g\phi)\le C_1\|\phi\|\theta_{k+1}^k.
\]
\end{lemma}
\begin{proof}
We have $\var_0(\mathscr{L}_g\phi)\le2\|\mathscr{L}_g\phi\|_{\infty}\le2Ne^{\max\Re g}\|\phi\|_{\infty}\le2Nb_1\|\phi\|$.
Let $k\in\N$.
By Lemma \ref{lem:302}, we have $\var_{k}(\mathscr{L}_{g}\phi)\le Nb_1\left\{3\var_{k+1}(g)\|\phi\|_{\infty}+\var_{k+1}(\phi)\right\}$. 
We also have $\var_{k+1}(g)\le b_2\theta_{k+1}^{k+1}\le Cb_2\theta_{k+1}^k$
and $\var_{k+1}(\phi)\le\|\phi\|\theta_{k+2}^{k+1}\le C\|\phi\|\theta_{k+1}^{k}$.
Hence, the assertion holds for $C_1=\max(2Nb_1,Nb_1(3Cb_2+C))$.
\end{proof}

For $g\in V$ and $m,q\in\N$, we define the two operators $K_{g,m},K_{g,m}^{(q)}$ by 
\[
K_{g,m}=\mathscr{L}_g\circ E_m,\quad K_{g,m}^{(q)}=\mathscr{L}_g^q-(\mathscr{L}_g-K_{g,m})^q.
\]
Notice that $K_{g,m}^{(1)}=K_{g,m}$.

\begin{lemma}
\label{lem:072}
Let $g\in V$. 
If there exists $b>0$ such that $\var_k(g)\le b\theta_k^k$ for $k\in\N$, then the following three assertions hold:
\begin{itemize}
\item[(i)]$\mathscr{L}_g(\mcB)\subset\mcB$.
\item[(ii)]$K_{g,m}^{(q)}(\mcB)\subset\mcB$ for $m,q\in\N$.
\item[(iii)]$\rank K_{g,m}^{(q)}\le q\rank E_m$ for $m,q\in\N$.
\end{itemize}
\end{lemma}
\begin{proof}
(i) follows from Lemma \ref{lem:035} immediately. 

(ii)\ Thanks to (i), it is enough to show that $K_{g,m}(\mcB)\subset\mcB$. 
From Corollary \ref{cor:488} and (i), we have $K_{g,m}(\mcB)=\mathscr{L}_g(E_m(\mcB))\subset\mathscr{L}_g(\mcB)\subset\mcB$.

(iii)\ We have $K_{g,m}^{(q)}=\sum_{k=0}^{q-1}\mathscr{L}_g^{q-1-k}K_{g,m}(\mathscr{L}_g-K_{g,m})^k$. 
For $k\in\{0,\dots,q-1\}$, $\rank\mathscr{L}_g^{p-1-k}K_{g,m}(\mathscr{L}_g-K_{g,m})^k\le\rank K_{g,m}\le\rank E_m$. 
Thus, we obtain (iii).
\end{proof}

The following lemma plays a key role in the proof of Theorem \ref{thm:919}.

\begin{lemma}
\label{lem:216}
There exists $C_2>0$, depending only on $b_1$ and $b_2$, such that the following two inequalities hold for $m,q\in\N$ with $\theta_{m+1}\le1$ and $g\in V$ satisfying (\ref{condi:029}):
\begin{align}
&\|\mathscr{L}_g\|_{\mcB\to\mcB}\le C_2,
\label{ineq:278}\\
&\|\mathscr{L}_g^q-K^{(q)}_{g,m}\|_{\mcB\to\mcB}\le C_2^q\theta_{m+1}^q.
\label{ineq:094}
\end{align}
\end{lemma}
\begin{proof}
We prove (\ref{ineq:278}). 
Let $C_1$ be as in Lemma \ref{lem:035}.
Take $\phi\in\mcB$ so that $\|\phi\|\le1$. 
Then, $\|\mathscr{L}_g\phi\|_{\infty}\le Ne^{\max\Re g}\|\phi\|_{\infty}\le Nb_1\|\phi\|\le Nb_1$. 
Moreover, from Lemma \ref{lem:035}, $\var_k(\mathscr{L}_g\phi)\le C_1\theta_{k+1}^k$ for $k\in\N\cup\{0\}$. 
Thus, (\ref{ineq:278}) holds for $C_2=Nb_1+C_1$.

We prove (\ref{ineq:094}). 
Since $\mathscr{L}_g^q-K^{(q)}_{g,m}=(\mathscr{L}_g-K_{g,m})^q$, we may prove (\ref{ineq:094}) only for $q=1$. 
Take $\phi\in\mcB$ so that $\|\phi\|\le1$. 

First, from Lemma \ref{lem:046} (i), $\|\mathscr{L}_g(\phi-E_m\phi)\|_{\infty}\le Ne^{\max\Re f}\|\phi-E_m\phi\|_{\infty}\le Nb_1\theta_{m+1}^m$. 
This and $\theta_{m+1}\le1$ imply
\begin{equation}
\label{ineq:212}
\|\mathscr{L}_g(\phi-E_m\phi)\|_{\infty}\le Nb_1\theta_{m+1}.
\end{equation}

Next, we show that
\begin{equation}
\label{ineq:306}
\var_{k}(\mathscr{L}_g(\phi-E_m\phi))\le Nb_1(3Cb_2+2)\theta_{m+1}\theta_{k+1}^k,\quad k\in\N\cup\{0\}.
\end{equation}
From (\ref{ineq:212}), $\var_0(\mathscr{L}_g(\phi-E_m\phi))\le2\|\mathscr{L}_g(\phi-E_m\phi)\|_{\infty}\le2Nb_1\theta_{m+1}$.
Let $k\in\N$. 
From Lemmas \ref{lem:302} and \ref{lem:046} (i),
\begin{align*}
&\var_{k}(\mathscr{L}_g(\phi-E_m\phi))\\
&\qquad\le Ne^{\max\Re g}\left\{3\var_{k+1}(g)\|\phi-E_m\phi\|_{\infty}+\var_{k+1}(\phi-E_m\phi)\right\}\\
&\qquad\le Nb_1\{3Cb_2\theta_{k+1}^{k}\theta_{m+1}^m+\var_{k+1}(\phi-E_m\phi)\}.
\end{align*}
Thus, it is enough to show that
\begin{equation}
\label{ineq:711}
\var_{k+1}(\phi-E_m\phi)\le2\theta_{m+1}\theta_{k+1}^k.
\end{equation}
From Lemma \ref{lem:046} (ii), if $k+1\ge m$, then $\var_{k+1}(\phi-E_m\phi)\le\theta_{k+2}^{k+1}=\theta_{k+2}\theta_{k+2}^{k}\le\theta_{m+1}\theta_{k+1}^k$,
and if $k+1<m$,
then $\var_{k+1}(\phi-E_m\phi)\le2\theta_{m+1}^m\le2\theta_{m+1}^{m-k}\theta_{m}^k\le2\theta_{m+1}\theta_{k+1}^k$.
Thus, (\ref{ineq:711}) holds.

By (\ref{ineq:212}) and (\ref{ineq:306}), we conclude that (\ref{ineq:094}) holds for $C_2=3Nb_1(Cb_2+1)$.
\end{proof}

We need the next lemma to prove Theorem \ref{thm:925}.

\begin{lemma}
\label{lem:950}
There exists $C_3>0$, depending only on $b_1$ and $b_2$, such that 
$\mathscr{L}_{g_m}(\mcB)\subset\mcB$ and the following inequality holds
for $m\in\N$ with $\theta_m\le1$ and $g\in V$ satisfying (\ref{condi:029}):
\[
\|\mathscr{L}_g-\mathscr{L}_{g_m}\|_{\mcB\to\mcB}\le C_3\theta_m,
\]
where $g_m=E_mg$.
\end{lemma}
\begin{proof}
From (i) and (ii) in Lemma \ref{lem:681}, $g_m$ also satisfies (\ref{condi:029}). Thus, $\mathscr{L}_{g_m}(\mcB)\subset\mcB$ follows form Lemma \ref{lem:072} (i).

We show that the inequality holds for $C_3=3Nb_1b_2(C+3)$. 
Take $\phi\in\mcB$ so that $\|\phi\|\le1$. 
It is enough to show the following two inequalities:
\begin{align}
&\|(\mathscr{L}_g-\mathscr{L}_{g_m})\phi\|_{\infty}\le3Nb_1b_2\theta_m,
\label{ineq:112}\\
&\var_{k}((\mathscr{L}_g-\mathscr{L}_{g_m})\phi)\le3Nb_1b_2(C+2)\theta_m\theta_{k+1}^k,\quad k\in\N\cup\{0\}.
\label{ineq:610}
\end{align}

To prove (\ref{ineq:112}) and (\ref{ineq:610}), we first show that
\begin{equation}
\|e^{g}-e^{g_m}\|_{\infty}\le3b_1b_2\theta_m^m.
\label{ineq:254}
\end{equation}
From (\ref{ineq:256}) and Lemma \ref{lem:681} (i), $\|e^{g}-e^{g_m}\|_{\infty}\le3b_1\|g-g_m\|_{\infty}$. 
If $\omega,\omega^{\prime}\in\Sigma_{{\bf A}}^+$ satisfy $\omega|m=\omega^{\prime}|m$, then $|g(\omega)-g(\omega^{\prime})|\le b_2\theta_m^m$, and hence, (\ref{ineq:254}) follows from the same argument as that in the proof of Lemma \ref{lem:681} (iii).

We prove (\ref{ineq:112}). 
Let $\omega\in\Sigma_{{\bf A}}^+$. Then, by (\ref{ineq:254}), we have
\[
|((\mathscr{L}_g-\mathscr{L}_{g_m})\phi)(\omega)|\le\sum_{i:{\bf A}(i\omega_0)=1}|e^{g(i\omega)}-e^{g_m(i\omega)}||\phi(i\omega)|\le3Nb_1b_2\theta_m^m.
\]
Moreover, $\theta_m^m\le\theta_m$ since $\theta_m\le1$.
Thus, (\ref{ineq:112}) holds.

We prove (\ref{ineq:610}). 
From (\ref{ineq:112}), $\var_0((\mathscr{L}_g-\mathscr{L}_{g_m})\phi)\le2\|(\mathscr{L}_g-\mathscr{L}_{g_m})\phi\|_{\infty}\le6Nb_1b_2\theta_m$. 
Let $k\in\N$ and let $\omega,\omega^{\prime}\in\Sigma_{{\bf A}}^+$ satisfy $\omega|k=\omega^{\prime}|k$. 
Since $\omega_0=\omega^{\prime}_0$, we have $\{i:{\bf A}(i\omega_0)=1\}=\{i:{\bf A}(i\omega_0^{\prime})=1\}$. 
For $i$ with ${\bf A}(i\omega_0)=1$, we set $I_1(i)=|e^{g(i\omega^{\prime})}-e^{g_m(i\omega^{\prime})}|\ |\phi(i\omega)-\phi(i\omega^{\prime})|$ and $I_2(i)=|e^{g(i\omega)}-e^{g_m(i\omega)}-(e^{g(i\omega^{\prime})}-e^{g_m(i\omega^{\prime})})|\ |\phi(i\omega)|$. 
Then,
\[
|((\mathscr{L}_g-\mathscr{L}_{g_m})\phi)(\omega)-((\mathscr{L}_g-\mathscr{L}_{g_m})\phi)(\omega^{\prime})|\le\sum_{i:{\bf A}(i\omega_0)=1}I_1(i)+I_2(i).
\]
By (\ref{ineq:254}), we have
\[
I_1(i)\le3b_1b_2\theta_m^m\theta_{k+2}^{k+1}\le 3Cb_1b_2\theta_m^m\theta_{k+1}^k.
\]
On the  other hand, if $k+1\ge m$, then, by (\ref{ineq:256}), $I_2(i)=|e^{g(i\omega)}-e^{g(i\omega^{\prime})}|\le3e^{\max\Re g}|g(i\omega)-g(i\omega^{\prime})|\le3b_1b_2\theta_{k+1}^{k+1}\le3b_1b_2\theta_m\theta_{k+1}^{k}$,
and if $k+1<m$, then, by (\ref{ineq:254}), $I_2(i)\le2\|e^{g}-e^{g_m}\|_{\infty}\le6b_1b_2\theta_m^m\le6b_1b_2\theta_m\theta_{k+1}^k$.
Thus,
\[
I_2(i)\le6b_1b_2\theta_m\theta_{k+1}^k,
\]
and hence, $I_1(i)+I_2(i)\le3b_1b_2(C+2)\theta_m\theta_{k+1}^k$.
Thus, (\ref{ineq:610}) holds.
\end{proof}

\section{Proof of Theorem \ref{thm:919}}
\label{sec:main01}

In this section, we prove Theorem \ref{thm:919}, which is the first main result of this paper.

\begin{proof}[Proof of Theorem \ref{thm:919}]
(i)\ $\mathscr{L}_f(\mcB)\subset\mcB$ follows from Lemma \ref{lem:072} (i). 
We show the compactness of $\mathscr{L}_f:\mcB\to\mcB$. 
By (\ref{ineq:094}), we have $\lim_{m\to\infty}\|\mathscr{L}_f-K_{f,m}\|_{\mcB\to\mcB}=0$. 
Since $K_{f,m}$ is a finite-rank operator, $\mathscr{L}_f$ is a compact operator.

(ii)\ Without loss of generality, we may assume that $\theta_m>0$ for any $m\in\N$. 
Then, from Lemma \ref{lem:293} (ii), 
\begin{equation}
\label{eq:392}
\bigcup_{m\ge0}L_m\subset\mcB.
\end{equation}

It is enough to show the following claim:

\begin{claim}
\label{cla:404}
Let $\rho>0$. 
Then, $\Lambda_f\cap\{|\lambda|>\rho\}=\sigma(\mathscr{L}_f:\mcB\to\mcB)\cap\{|\lambda|>\rho\}$. 
Moreover, for $\lambda\in\Lambda_f$ with $|\lambda|>\rho$, the corresponding multiplicities coincide with each other. 
\end{claim}

For $\theta\in(0,1)$, we denote by $V_{\theta}$ the completion of $V$ by the norm $\|\cdot\|_{\theta}$. 
Clearly, $V_{\theta}\subset F_{\theta}$. Hence, from Lemma \ref{lem:405}, $V=\bigcap_{\theta\in(0,1)}V_{\theta}$. 
Moreover, $\mathscr{L}_f(V_{\theta})\subset V_{\theta}$ since $\mathscr{L}_f(V)\subset V$.
Therefore, it is enough to show the following claim:
\begin{claim}
\label{cla:435}
Let $\rho>0$. 
For $\theta\in(0,1)$ with $\theta e^{P(\Re f)}<\rho$ and $\lambda\in\sigma(\mathscr{L}_f:V_{\theta}\to V_{\theta})$ with $|\lambda|>\rho$, $\lambda$ is an eigenvalue of $\mathscr{L}_f:\mcB\to\mcB$, and hence, $\lambda$ is that of $\mathscr{L}_f:V_{\theta}\to V_{\theta}$. 
Moreover, the corresponding multiplicities coincide with each other.
\end{claim}

From \cite[Theorem 1]{Pollicott Invent 1986}, $r_{\mathrm{ess}}(\mathscr{L}_f:V_{\theta}\to V_{\theta})\le r_{\mathrm{ess}}(\mathscr{L}_f:F_{\theta}\to
F_{\theta})=\theta e^{P(\Re f)}<\rho$. 
Here, for a Banach space $E$ and a bounded linear operator $T$ on $E$, $r_{\mathrm{ess}}(T:E\to E)$ denotes the essential spectral radius of $T$, that is,
\begin{align*}
r_{\mathrm{ess}}(T:E\to E)=\inf\{r\ge0:\>& \text{any $\lambda\in\sigma(T:E\to E)$ with $|\lambda|>r$ is an}\\
                                            & \text{isolated eigenvalue with finite multiplicity}\}.
\end{align*}
From Lemma \ref{lem:681} (iv), $\bigcup_{m\ge0}L_m$ is dense in $V_{\theta}$. 
Therefore, from (\ref{eq:392}), $\mcB$ is dense in $V_{\theta}$. 
Thus, Claim \ref{cla:435} follows from \cite[Lemma A.1]{Baladi Contemp.Math 2008}.
\end{proof}

\section{Proof of Theorem \ref{thm:925}}
\label{sec:main02}

In this section, we prove Theorem \ref{thm:925}, which is the second main result of this paper.

Let $f\in V$ and let $r\in(0,1)$ satisfy (\ref{cond:209}). 
Take $D>0$ so that 
\[
\var_m(f)^{1/m}\le Dr^m,\quad m\in\N.
\]
In this section, we set
\begin{equation}
\label{eq:204}
\theta_m=Dr^m,\quad m\in\N.
\end{equation}

When $\theta_m$ is of the form (\ref{eq:204}), we can obtain a slightly better estimate than (\ref{ineq:094}). 

\begin{lemma}
\label{lem:807}
There exists $C_4>0$, depending only on $b_1$ and $b_2$, such that the following inequality holds for $m,q\in\N$ with $m\ge2,Dr^{m+1}\le1$ and $g\in V$ satisfying (\ref{condi:029}):
\[
\|\mathscr{L}_g^q-K^{(q)}_{g,m}\|_{\mcB\to\mcB}\le C_4^q(r^{2m})^q.
\]
\end{lemma}
\begin{proof}
The outline of the proof is the same as that of the proof of (\ref{ineq:094}). 
We may prove the inequality only for $q=1$. 
Notice that $C$ in (\ref{eq:224}) is equal to $Dr$.

Take $\phi\in\mcB$ so that $\|\phi\|\le1$. 
From Lemma \ref{lem:046} (i), $\|\mathscr{L}_g(\phi-E_m\phi)\|_{\infty}\le Nb_1\theta_{m+1}^m$.
Since $m\ge2$ and $\theta_{m+1}=Dr^{m+1}\le1$, we have $\theta_{m+1}^m\le\theta_{m+1}^2=D^2r^{2m+2}$, and hence, 
\begin{equation}
\label{ineq:252}
\|\mathscr{L}_g(\phi-E_m\phi)\|_{\infty}\le D^2Nb_1r^{2m+2}.
\end{equation}
Thus, we may show that
\[
\var_{k}(\mathscr{L}_g(\phi-E_m\phi))\le DNb_1(3D^2r^3b_2+\max(1,2Dr^2))r^{2m}\theta_{k+1}^k,\quad k\in\N\cup\{0\}.
\]

From (\ref{ineq:252}), $\var_0(\mathscr{L}_g(\phi-E_m\phi))\le2\|\mathscr{L}_g(\phi-E_m\phi)\|_{\infty}\le2D^2Nb_1r^{2m+2}$. Let $k\in\N$. 
From Lemmas \ref{lem:302} and \ref{lem:046} (i),
\[
\var_{k}(\mathscr{L}_g(\phi-E_m\phi))\le Nb_1\{3D^3r^3b_2r^{2m}\theta_{k+1}^k+\var_{k+1}(\phi-E_m\phi)\}.
\]
Thus, it is enough to show that
\begin{equation}
\label{ineq:733}
\var_{k+1}(\phi-E_m\phi)\le D\max(1,2Dr^2)r^{2m}\theta_{k+1}^k.
\end{equation}
We first assume $k+1\ge m$. 
From Lemma \ref{lem:046} (ii), $\var_{k+1}(\phi-E_m\phi)\le\theta_{k+2}^{k+1}$.
Moreover,
\[
\theta_{k+2}^{k+1}=\left(\frac{\theta_{k+2}}{\theta_{k+1}}\right)^{k+1}\theta_{k+1}\theta_{k+1}^k=Dr^{2(k+1)}\theta_{k+1}^k\le Dr^{2m}\theta_{k+1}^k.
\]
Thus, (\ref{ineq:733}) holds for $k+1\ge m$.
We next assume $k+1<m$. 
Then, from Lemma \ref{lem:046} (ii), $ \var_{k+1}(\phi-E_m\phi)\le2\theta_{m+1}^{m-k}\theta_{m}^k\le2\theta_{m+1}^2\theta_{k+1}^k=2D^2r^{2m+2}\theta_{k+1}^k$.
Thus, (\ref{ineq:733}) holds for $k+1<m$. 
\end{proof}

For $x\in\R$, $\lfloor x\rfloor$ denotes the largest integer less than or equal to $x$. 
Recall, from (\ref{def:785}) in Section \ref{sec:introduction}, the definition of approximation numbers of a bounded linear operator acting on a Banach space.
We estimate the approximation numbers of transfer operators acting on the Banach space $\mcB$.

\begin{lemma}
\label{lem:007}
Let $C_4$ be as in Lemma \ref{lem:807}. 
For $q\in\N$ and $R>e^{h_{\mathrm{top}}(\sigma_{{\bf A}})}$, there exists $M_1\in\N$, depending only on $q$ and $R$, such that the following inequality holds for $m\ge M_1$ and $g\in V$ satisfying (\ref{condi:029}): 
\[
a_{\lfloor R^m\rfloor+1}(\mathscr{L}_g^q)\le C_4^q(r^{2m})^q.
\]
\end{lemma}
\begin{proof}
The equation $\rank E_m=\#\{w:\text{${\bf A}$-admissible and $|w|=m$}\}$ implies that $(\rank E_m)^{1/m}\to e^{h_{\mathrm{top}}(\sigma_{{\bf A}})}$ as $m\to\infty$. 
Thus, there exists $M_1\ge2$ such that $q\rank E_m\le R^m$ and $Dr^{m+1}\le1$ for $m\ge M_1$. 
Let $m\ge M_1$. 
From Lemma \ref{lem:072} (iii), $\rank K_{g,m}^{(q)}\le R^m<\lfloor R^m\rfloor+1$. 
Hence, by Lemma \ref{lem:807}, we have $a_{\lfloor R^m\rfloor+1}(\mathscr{L}_g^q)\le\|\mathscr{L}_g^q-K^{(q)}_{g,m}\|_{\mcB\to\mcB}\le C_4^q(r^{2m})^q$, as desired.
\end{proof}

\begin{corollary}
\label{cor:517}
Let $C_4$ be as in Lemma \ref{lem:807}. 
For $q\in\N$ and $R>e^{h_{\mathrm{top}}(\sigma_{{\bf A}})}$, there exists $N_1\in\N$, depending only on $q$ and $R$, such that the following inequality holds for $n\ge N_1$ and $g\in V$ satisfying (\ref{condi:029}): 
\[
a_{n}(\mathscr{L}_g^q)\le\left(\frac{C_4}{r^2}\right)^q\left(\frac{1}{n-1}\right)^{\frac{2q}{-\log_rR}}.
\]
\end{corollary}
\begin{proof}
Let $M_1$ be as in Lemma \ref{lem:007}. 
Take $N_1\in\N$ so that $\lfloor \log_R(n-1)\rfloor\ge M_1$ for $n\ge N_1$. Let $n\ge N_1$. 
Then, $n\ge\lfloor R^{\lfloor \log_R(n-1)\rfloor}\rfloor+1$ since $n>n-1\ge R^{\lfloor \log_R(n-1)\rfloor}$. 
Put $m=\lfloor \log_R(n-1)\rfloor$. 
From Lemma \ref{lem:007}, $a_n(\mathscr{L}_g^q)\le C_4^q(r^{2m})^q=(C_4/r^2)^q(r^{2(m+1)})^q$.
Since $m+1\ge\log_R(n-1)$, we have $r^{2(m+1)}\le r^{2\log_R(n-1)}=(n-1)^{2/\log_rR}$. 
Thus, the desired result follows. 
\end{proof}

Recall from Section \ref{sec:introduction} the definition of the operator ideal $\mathfrak{L}_p^{(a)}(E)$.
The following corollary plays a key role in the proof of Theorem \ref{thm:925}.

\begin{corollary}
\label{cor:029}
Let $g\in V$ satisfy (\ref{condi:029}). 
For $p>0$ with (\ref{cond:468}) and $m,q\in\N$ with $p\ge q$, the following two assertions hold (we write $g_m=E_mg$):
\begin{itemize}
\item[(i)]$\mathscr{L}_g,\mathscr{L}_{g_m}\in\mathfrak{L}_p^{(a)}(\mcB)$.
\item[(ii)]$\mathscr{L}_g^q,\mathscr{L}_{g_m}^q\in\mathfrak{L}_1^{(a)}(\mcB)$ and $\sum_{n=1}^{\infty}a_n(\mathscr{L}_g^q-\mathscr{L}_{g_m}^q)\to0$ as $m\to\infty$.
\end{itemize}
\end{corollary}
\begin{proof}
We prove (i) and the former part of (ii).
From (i) and (ii) in Lemma \ref{lem:681}, $g_m$ also satisfies (\ref{condi:029}). 
Thus, it is enough to prove the assertions only for $g$. 
Take $R>e^{h_{\mathrm{top}}(\sigma_{{\bf A}})}$ with $r^{2p}R<1$. 
Then,
\begin{equation}
\label{ineq:618}
-\frac{2q}{\log_rR}\ge-\frac{2p}{\log_rR}>1.
\end{equation}
Hence, $\mathscr{L}_g\in\mathfrak{L}_p^{(a)}(\mcB)$ and $\mathscr{L}_g^q\in\mathfrak{L}_1^{(a)}(\mcB)$ follow from Corollary \ref{cor:517}.

We prove the latter part of (ii).
Let  $N_1\in\N$ be as in Corollary \ref{cor:517}. 
Let $\epsilon>0$. 
By (\ref{ineq:618}), we can take $n_0\ge N_1$ so that $\sum_{n\ge n_0}(C_4/r^2)^q(n-1)^{2q/\log_rR}<\epsilon$.
Moreover, by Lemma \ref{lem:950}, we can take $m_0\in\N$ so that $\theta_{m+1}\le1$ and $n_0\|\mathscr{L}_g^q-\mathscr{L}_{g_m}^q\|_{\mcB\to\mcB}<\epsilon$ for $m\ge m_0$.

Let $m\ge m_0$. 
We show $\sum_{n=1}^{\infty}a_n(\mathscr{L}_g^q-\mathscr{L}_{g_m}^q)<6\epsilon$. 
We easily have
\[
\sum_{1\le n<2n_0}a_n(\mathscr{L}_g^q-\mathscr{L}_{g_m}^q)\le2n_0a_1(\mathscr{L}_g^q-\mathscr{L}_{g_m}^q)=2n_0\|\mathscr{L}_g^q-\mathscr{L}_{g_m}^q\|_{\mcB\to\mcB}<2\epsilon.
\]
On the other hand, we have
\[
\sum_{n\ge2n_0}a_n(\mathscr{L}_g^q-\mathscr{L}_{g_m}^q)\le\sum_{l\ge n_0}\sum_{n:2l\le n<2(l+1)}a_n(\mathscr{L}_g^q-\mathscr{L}_{g_m}^q)\le2\sum_{l\ge n_0}a_{2l}(\mathscr{L}_g^q-\mathscr{L}_{g_m}^q).
\]
By \cite[Theorem 2.3.3]{Pietsch}, $a_{2l}(\mathscr{L}_g^q-\mathscr{L}_{g_m}^q)\le a_{2l-1}(\mathscr{L}_g^q-\mathscr{L}_{g_m}^q)\le a_{l}(\mathscr{L}_g^q)+a_{l}(\mathscr{L}_{g_m}^q)$. 
Thus, by Corollary \ref{cor:517}, we have
\[
\sum_{n\ge2n_0}a_n(\mathscr{L}_g^q-\mathscr{L}_{g_m}^q)\le2\left\{\sum_{l\ge n_0}a_l(\mathscr{L}_g^q)+\sum_{l\ge n_0}a_l(\mathscr{L}_{g_m}^q)\right\}<4\epsilon,
\]
as desired.
\end{proof}

We are ready to prove Theorem \ref{thm:925}.

\begin{proof}[Proof of Theorem \ref{thm:925}]
(i) follows from Corollary \ref{cor:029} (i) and (I) in Section \ref{sec:introduction}.

(ii)\ Let $f_m=E_mf$. 
Then, $\sum_{\omega\in\Per_q(\sigma_{{\bf A}})}e^{S_qf(\omega)}=\lim_{m\to\infty}\sum_{\omega\in\Per_q(\sigma_{{\bf A}})}e^{S_qf_m(\omega)}$.
Since $f_m$ is locally constant, $\sum_{\omega\in\Per_q(\sigma_{{\bf A}})}e^{S_qf_m(\omega)}=\sum_{n=1}^{\infty}\lambda_n(f_m)^q$.
Moreover, 
$\lim_{m\to\infty}\sum_{n=1}^{\infty}\lambda_n(f_m)^q=\sum_{n=1}^{\infty}\lambda_n(f)^q$ follows from Corollary \ref{cor:029} (ii) and (II) in Section \ref{sec:introduction}. 
Thus, (\ref{eq:008}) holds.

(iii)\ From \cite[Theorem 2.6.5]{Boas}, the infinite product $P(z)=\prod_{n=1}^{\infty}E(z\lambda_n(f),k_0)$ converges uniformly on any compact set of $\C$.
Therefore, it is enough to show the following equality for $z\in\C$ with sufficiently small $|z|$:
\[
P(z)=\exp\left(-\sum_{q=k_0+1}^{\infty}\frac{z^q}{q}\sum_{\omega\in\Per_q(\sigma_{{\bf A}})}e^{S_qf(\omega)}\right).
\]
We denote by $\Log$ the principle branch of the complex logarithm. 
By $\Log(1-z)=-\sum_{n\ge1}z^n/n$ for $|z|<1$, we have $P(z)=\exp(-\sum_{n=1}^{\infty}\sum_{q=k_0+1}^{\infty}\frac{z^q}{q}\lambda_n(f)^q)$ for $z\in\C$ with sufficiently small $|z|$.
Since $\sum_{n=1}^{\infty}|\lambda_n(f)|^{k_0+1}<\infty$, we can exchange the order of the summation and we have $P(z)=\exp(-\sum_{q=k_0+1}^{\infty}\frac{z^q}{q}\sum_{n=1}^{\infty}\lambda_n(f)^q)=\exp(-\sum_{q=k_0+1}^{\infty}\frac{z^q}{q}\sum_{\omega\in\Per_q(\sigma_{{\bf A}})}e^{S_qf(\omega)})$, as desired.
\end{proof}

We show that, for each $r\in(0,1)$, there exists a non-locally constant $f\in V$ satisfying (\ref{cond:209}). 

\begin{example}
\label{exa:004}
\textrm{
Let ${\bf A}=\left(\begin{smallmatrix}
1&1\\
1&1
\end{smallmatrix}\right)$. 
Take two sequences $\{\theta_m^{(1)}\}_{m\in\N\cup\{0\}},\,\{\theta_m^{(2)}\}_{m\in\N\cup\{0\}}$ so that 
$\theta_m^{(1)}>\theta_m^{(2)}>0,\ m\in\N\cup\{0\}$, and 
\[
1>\theta_m^{(1)} \downarrow 0,\qquad 1>\theta_m^{(2)} \downarrow0,\qquad 
\limsup_{m\to\infty}\frac{\theta_m^{(2)}}{\theta_m^{(1)}}<1.
\]
We define $f:\Sigma_{{\bf A}}^+\to\R$ by 
$f(\omega)=\sum_{m\ge0}\{\theta_m^{(\omega_m)}\}^{1/m}$. 
Then, it is easy to see that $\var_m(f)^{1/m}\asymp\theta_m^{(1)}$, that is, there exists $C>0$ such that $C^{-1}\le\var_m(f)^{1/m}/\theta_m^{(1)}\le C$ for $m\in\N$.
Hence, for 
$r\in(0,1)$, $\theta_m^{(1)}=r^m$ and $\theta_m^{(2)}=(1/2)\theta_m^{(1)}$, 
$f$ is a real-valued, non-locally constant function in $V$ such that $\var_m(f)^{1/m}\asymp r^m$.
}
\end{example}

\appendix
\section{Transfer operators on $V$}
\label{appendix:transfer operator}

A metrizable topological vector space is said to be \emph{complete} if every Cauchy sequence converges. 
Note that our space $V$ is metrizable and complete since $V$ is a Fr\'{e}chet space (see Proposition \ref{Prop:715}).
We recall the following definition of a compact operator on a metrizable and complete topological vector space. 

\begin{definition}
Let $X$ be a metrizable and complete topological vector space and $T$ a continuous linear operator on $X$. 
We say that $T$ is a \emph{compact operator} if the closure $\overline{T(N)}$ of the image $T(N)$ is compact for some neighborhood $N$ of zero.
\end{definition}

In this appendix, we prove the following two theorems:

\begin{theorem}
\label{thm:904}
$\mathscr{L}_{f}:V\to V$ is not a compact operator for any $f\in V$.
\end{theorem}

\begin{theorem}
\label{thm:205}
If ${\bf A}=\left(\begin{smallmatrix}
1&1\\
1&1
\end{smallmatrix}\right)$, then there exists a real-valued $f\in V$ such that $\sum_{n\ge1}|\lambda_{n}(f)|=\infty$. 
\end{theorem}

First, we prove Theorem \ref{thm:205}.

\begin{proof}[Proof of Theorem \ref{thm:205}]
\cite[Proposition 4.1]{Jezequel 2018} implies that there exists a real-valued $f\in V$ such that $\zeta_{f}(z)^{-1}=1-2z-z(1-z)\sin z$. 
For $x\in\R$, we write $F(x)=(1-2x)/\{x(1-x)\}$ and $G(x)=\sin x-F(x)$. 
Since $F(x)>0$ for $x>1$ and $\lim_{x\to\infty}F(x)=0$, we find that $G(2n\pi)<0$ and $G(2n\pi+(\pi/2))>0$ for sufficiently large $n\in\N$. 
Thus, for sufficiently large $n\in\N$, there exists $\eta_n\in(2n\pi,\ 2n\pi+\frac{\pi}{2})$ such that $G(\eta_n)=0$. 
Therefore, we have
\[
\sum_{n\ge1}|\lambda_n(f)|\ge\sum_{n:\,\text{large}}\frac{1}{\eta_n}\ge\sum_{n:\,\text{large}}\frac{1}{2n\pi+\frac{\pi}{2}}=\infty,
\]
as desired
\end{proof}

Next, we prove Theorem \ref{thm:904}.
For the sake of simplicity, we write
\[
\var_{m}^{\theta}(\phi)=\frac{\var_{m}(\phi)}{\theta^m}
\]
for $\theta\in(0,1),\phi\in F_{\theta}$ and $m\in\N\cup\{0\}$.

To prove Theorem \ref{thm:904}, we need the following lemma:

\begin{lemma}
Let $A\subset V$ be a neighborhood of zero. 
Then there exists an open neighborhood 
$G\subset A$ of zero satisfying the following two conditions:
\begin{align}
&\frac{\phi}{2}1_{[i]}\in G \quad \mbox{for}\ \phi\in G\ \mbox{and}\ i\in\{1,\dots,N\}, 
\label{eq:111}\\
&\mbox{there exists}\ \theta\in(0,1)\ \mbox{such that}\ \sup_{\phi\in G}\|\phi\|_{\theta}=\infty.
\label{eq:222}
\end{align}
\end{lemma}
\begin{proof}
The collection of the sets of the form $\{\phi\in V:\|\phi\|_{\theta}<\epsilon\},\ \theta\in(0,1),\ \epsilon>0$, is a fundamental system of neighborhoods of zero. 
Thus, there exist $\theta_{0}\in(0,1)$ and $\epsilon>0$ such that the set $G=\{\phi\in V:\|\phi\|_{\theta_{0}}<\epsilon\}$ is contained in $A$. 
We prove (\ref{eq:111}) and (\ref{eq:222}) for $G$.

Let $\phi\in G$ and $i\in\{1,...,N\}$. 
If $\omega_0=\omega_0^{\prime}$, then $\phi1_{[i]}(\omega)-\phi1_{[i]}(\omega^{\prime})=\phi(\omega)-\phi(\omega^{\prime})\ (\omega_{0}=i),\ =0\ (\omega_{0}\neq i)$, and hence, for $m\in\N$, we have $\var_m^{\theta_{0}}(\phi1_{[i]})\le[\phi]_{\theta_0}$. 
Also, if $\omega_{0}\neq\omega_{0}^{\prime}$, then $\phi1_{[i]}(\omega)-\phi1_{[i]}(\omega^{\prime})=\phi(\omega)\ (\omega_{0}=i),\ =0\ (\text{$\omega_{0}\neq i$ and $\omega_{0}^{\prime}\neq i$})$, and hence, we have $\var_{0}^{\theta_{0}}(\phi1_{[i]})\le\|\phi\|_{\infty}$. 
Combining, we have
\[
\|\phi1_{[i]}\|_{\theta_{0}}=\|\phi1_{[i]}\|_{\infty}+\sup_{m\ge0}\var_m^{\theta_{0}}(\phi1_{[i]})
\le\|\phi\|_{\infty}+\max(\|\phi\|_{\infty},\,[\phi]_{\theta_{0}})\le2\|\phi\|_{\theta_{0}}.
\]
Thus, $\|(\phi/2)1_{[i]}\|_{\theta_{0}}\le\|\phi\|_{\theta_{0}}<\epsilon$, and hence, $(\phi/2)1_{[i]}\in G$. 
Therefore, (\ref{eq:111}) follows.

Lemma \ref{lem:203} ensures that there exist $i,j_1,j_2\in\{1,\dots,N\}$ such that $j_{1}\neq j_{2}$ and ${\bf A}(ij_{1})={\bf A}(ij_{2})=1$. 
Moreover, there exists $j\in\{1,\dots,N\}$ such that ${\bf A}(ji)=1$. 
Take an ${\bf A}$-admissible word $w$ such that $w_0=i$ and $w_{|w|-1}=j$. For $n\in\N$, we define
\[
\phi_{n}=\frac{\theta_{0}^{n|w|+1}\epsilon}{2}1_{[\underbrace{w\cdots w}_{n}ij_{1}]}.
\]
For $\theta\in(0,1)$, we have $\var_{n|w|+1}^{\theta}(1_{[\underbrace{w\cdots w}_{n}ij_{1}]})=1/\theta^{n|w|+1}$, and hence, 
\[
\|\phi_{n}\|_{\theta}
=\frac{\theta_{0}^{n|w|+1}\epsilon}{2}\left(1+\frac{1}{\theta^{n|w|+1}}\right).
\]
In particular, we have $\|\phi_n\|_{\theta_0}<\epsilon$. 
Therefore, $\phi_n\in G$. On the other hand, for $\theta\in(0,\theta_0)$, we see that $\|\phi_{n}\|_{\theta}\ge(\epsilon/2)(\theta_0/\theta)^{n|w|+1}\to\infty$ as $n\to\infty$, hence $\sup_{\phi\in G}\|\phi\|_{\theta}\ge\sup_{n\in\N}\|\phi_n\|_{\theta}=\infty$. 
Thus, (\ref{eq:222}) follows.
\end{proof}

\begin{proof}[Proof of Theorem \ref{thm:904}]
Fix a neighborhood $A\subset V$ of zero. 
We show that the following assertion holds:
\begin{equation}
\begin{aligned}
\label{fact:037}
&\text{there exists $\{\phi_{n}\}_{n\in\N}\subset A$ such that $\{\mathscr{L}_{f}\phi_{n}\}_{n\in\N}$ has}\\
&\text{no convergent subsequence}.
\end{aligned}
\end{equation}
Without loss of generality, we may assume that $A$ satisfies both (\ref{eq:111}) and (\ref{eq:222}) with $G$ replaced by $A$. 

First, we consider the case in which $\sup_{\phi\in A}\|\phi\|_{\infty}=\infty$. 
The inequality $\|\phi\|_{\infty}\le\sum_{i=1}^{N}\|\phi1_{[i]}\|_{\infty}$ implies that there exists $i\in\{1,\dots,N\}$ such that $\sup_{\phi\in A}\|\phi1_{[i]}\|_{\infty}=\infty$. 
Thus, we can take $\{\psi_{n}\}_{n\in\N}\subset A$ so that $\lim_{n\to\infty}\|(\psi_{n}/2)1_{[i]}\|_{\infty}=\infty$. 

We write $\phi_{n}=(\psi_{n}/2)1_{[i]}$. 
It is clear that $\{\phi_{n}\}_{n\in\N}\subset A$. 
Take $\omega^{(n)}\in\Sigma_{{\bf A}}^{+}$ so that $\|\phi_{n}\|_{\infty}=|\phi_{n}(\omega^{(n)})|$. 
Since $\|\phi_{n}\|_{\infty}>0$ implies $\omega_{0}^{(n)}=i$, we have
\begin{equation}
\label{eq:407}
\begin{aligned}
(\mathscr{L}_{f}\phi_{n})(\sigma_{{\bf A}}\omega^{(n)})=&\sum_{j:\,{\bf A}(j\omega_{1}^{(n)})=1}e^{f(j(\sigma_{{\bf A}}\omega^{(n)}))}\phi_{n}(j(\sigma_{{\bf A}}\omega^{(n)}))1_{[i]}(j(\sigma_{{\bf A}}\omega^{(n)}))\\
=&e^{f(i(\sigma_{{\bf A}}\omega^{(n)}))}\phi_{n}(i(\sigma_{{\bf A}}\omega^{(n)}))=e^{f(\omega^{(n)})}\phi_{n}(\omega^{(n)})
\end{aligned}
\end{equation}
for sufficiently large $n\in\N$. 
Hence,
\[
\|\mathscr{L}_{f}\phi_{n}\|_{\infty}\ge|(\mathscr{L}_{f}\phi_{n})(\sigma_{{\bf A}}\omega^{(n)})|=e^{\Re f(\omega^{(n)})}|\phi_{n}(\omega^{(n)})|\ge e^{\min\Re f}\|\phi_{n}\|_{\infty}\to\infty
\]
as $n\to\infty$. 
Therefore, (\ref{fact:037}) holds.

Next, we consider the case in which $\sup_{\phi\in A}\|\phi\|_{\infty}<\infty$.
We see that there exists  $\theta\in(0,1)$ such that $\sup_{\phi\in A}\|\phi\|_{\theta}=\infty$. 
Since $\|\phi\|_{\theta}=\|\phi\|_{\infty}+[\phi]_{\theta}$ and $\sup_{\phi\in A}\|\phi\|_{\infty}<\infty$, we have $\sup_{\phi\in A}[\phi]_{\theta}=\infty$. 
Moreover, for $m\in\N$, we have
\begin{align*}
\infty=\sup_{\phi\in A}[\phi]_{\theta}\le&\left(\sup_{\phi\in A}\sup_{0\le k\le m}\var_{k}^{\theta}(\phi)\right)+\left(\sup_{\phi\in A}\sup_{k>m}\var_{k}^{\theta}(\phi)\right)\\
\le&\frac{2}{\theta^m}\sup_{\phi\in A}\|\phi\|_{\infty}+\left(\sup_{\phi\in A}\sup_{k>m}\var_{k}^{\theta}(\phi)\right),
\end{align*}
and hence, we have $\sup_{\phi\in A}\sup_{k>m}\var_{k}^{\theta}(\phi)=\infty$. 
Therefore, for $k\in\N$, there exist $\omega^{(k)},\,\tilde{\omega}^{(k)}\in\Sigma_{{\bf A}}^{+}$ and $\psi_{k}\in A$ such that
\[
\omega^{(k)}\neq\tilde{\omega}^{(k)},\quad\lim_{k\to\infty}d_{\theta}(\omega^{(k)},\,\tilde{\omega}^{(k)})=0,\quad\lim_{k\to\infty}\frac{|\psi_{k}(\omega^{(k)})-\psi_{k}(\tilde{\omega}^{(k)})|}{d_{\theta}(\omega^{(k)},\,\tilde{\omega}^{(k)})}=\infty.
\]
We can choose $i\in\{1,\dots,N\}$ so that 
\[
\mathscr{N}=\{k\in\N:\omega_{0}^{(k)}=\tilde{\omega}_{0}^{(k)}=i\}
\]
is an infinite set. 
We write $\mathscr{N}=\{k_{1}, k_{2},\dots\}$, where $k_n<k_{n+1}$ for $n\in\N$. 
Notice that, for $n\in\N$, $d_{\theta}(\sigma_{{\bf A}}\omega^{(k_{n})},\,\sigma_{{\bf A}}\tilde{\omega}^{(k_{n})})>0$, and hence, $d_{\theta}(\omega^{(k_{n})},\tilde{\omega}^{(k_{n})})=\theta d_{\theta}(\sigma_{{\bf A}}\omega^{(k_{n})},\sigma_{{\bf A}}\tilde{\omega}^{(k_{n})})$.

We write $\phi_{n}=(\psi_{k_{n}}/2)1_{[i]}$. 
It is clear that $\{\phi_{n}\}_{n\in\N}\subset A$. 
By a calculation similar to (\ref{eq:407}), we have
\begin{align*}
(\mathscr{L}_{f}\phi_{n})(\sigma_{{\bf A}}&\omega^{(k_{n})})-(\mathscr{L}_{f}\phi_{n})(\sigma_{{\bf A}}\tilde{\omega}^{(k_{n})})\\
=&e^{f(\omega^{(k_{n})})}\{\phi_{n}(\omega^{(k_{n})})-\phi_{n}(\tilde{\omega}^{(k_{n})})\}-\phi_{n}(\tilde{\omega}^{(k_{n})})\{e^{f(\tilde{\omega}^{(k_{n})})}-e^{f(\omega^{(k_{n})})}\},
\end{align*}
and thus, we have
\begin{align*}
&\|\mathscr{L}_{f}\phi_{n}\|_{\theta}\ge\frac{|(\mathscr{L}_{f}\phi_{n})(\sigma_{{\bf A}}\omega^{(k_{n})})-(\mathscr{L}_{f}\phi_{n})(\sigma_{{\bf A}}\tilde{\omega}^{(k_{n})})|}{d_{\theta}(\sigma_{{\bf A}}\omega^{(k_{n})},\,\sigma_{{\bf A}}\tilde{\omega}^{(k_{n})})}\\
&\ge\theta\left\{\frac{|e^{f(\omega^{(k_{n})})}\{\phi_{n}(\omega^{(k_{n})})-\phi_{n}(\tilde{\omega}^{(k_{n})})\}|}{d_{\theta}(\omega^{(k_{n})},\,\tilde{\omega}^{(k_{n})})}-\frac{|\phi_{n}(\tilde{\omega}^{(k_{n})})\{e^{f(\omega^{(k_{n})})}-e^{f(\tilde{\omega}^{(k_{n})})}\}|}{d_{\theta}(\omega^{(k_{n})},\,\tilde{\omega}^{(k_{n})})}\right\}\\
&\ge\frac{\theta e^{\min\Re f}}{2}\,\frac{|\psi_{k_n}(\omega^{(k_{n})})-\psi_{k_n}(\tilde{\omega}^{(k_{n})})|}{d_{\theta}(\omega^{(k_{n})},\,\tilde{\omega}^{(k_{n})})}-\theta\left(\sup_{\phi\in A}\|\phi\|_{\infty}\right)\|e^{f}\|_{\theta}\to\infty
\end{align*}
as $n\to\infty$. 
Therefore, (\ref{fact:037}) holds.
\end{proof}


\section{Some properties of $V$}
\label{appendix:TVS}

We recall the following definitions of a nuclear operator and a nuclear space.

\begin{definition}
Let $E, F$ be Banach spaces and $T:E\to F$ a bounded linear operator. 
We say that $T$ is a \emph{nuclear operator} if $T$ can be written in the form $Tx=\sum_{n=1}^{\infty}\lambda_n\langle x,x^{\prime}_n\rangle y_n$, where the sequence $\{\lambda_n\}_{n\in\N}\subset\C$ is summable and both of the two sequences $\{x_{n}\}_{n\in\N}\subset E^{\prime}$ and $\{y_{n}\}_{n\in\N}\subset F$ are bounded. ($E^{\prime}$ denotes the dual Banach space of $E$.)
\end{definition}

\begin{definition}
Let $X$ be a locally convex Hausdorff topological vector space. 
We say that $X$ is a \emph{nuclear space} if, for every continuous seminorm $p$ on $X$, there exists a continuous seminorm $q$ on $X$ such that $p\le q$ and the natural embedding $\widehat{X}_{q}\to\widehat{X}_{p}$ is a nuclear operator.
Here, $\widehat{X}_{p}$ denotes the completion of $X/\ker p$ by $p$. 
\end{definition}

Let $E,F$ be Banach spaces and $T:E\to F$ a bounded linear operator. 
We extend the definition (\ref{def:785}) of the approximation numbers of $T$ to the case $E\neq F$. 
For $n\in\N$, the $n$-th \emph{approximation number} $a_n(T:E\to F)$ of $T$ is defined by 
\begin{align*}
a_n(T:E&\to F)\\
=&\inf\{\|T-A\|:A:E\to F\ \mbox{is a finite-rank operator with}\ \rank A<n\}.
\end{align*}

In this appendix, we prove the following two theorems:

\begin{theorem}
\label{thm:988}
$V$ is a nuclear space.
\end{theorem}

\begin{theorem}
\label{thm:311}
Let $R$ be the set of $\phi\in V$ such that $\phi$ is not cohomologous with any locally constant function, that is,
\begin{equation*}
\begin{aligned}
R=\{\phi\in V:\ &\phi-\varphi\neq\psi\circ\sigma_{{\bf A}}-\psi\ \mbox{for any}\\
&\mbox{locally constant function}\ \varphi\ \mbox{and continuous function}\ \psi\}.
\end{aligned}
\end{equation*}
Then, $R$ is a residual subset of $V$. 
In other words, there exists a sequence $V_1,V_2,\dots$ of open and dense subsets of $V$ such that $R\supset\bigcap_{m\in\N}V_m$. 
\end{theorem}

To prove Theorem \ref{thm:988}, we need the following lemma:

\begin{lemma}
\label{lem:343}
Let $\theta,\theta^{\prime}\in(0,1)$. If $\theta<\theta^{\prime}/N$, then the natural embedding $\iota:F_{\theta}\to F_{\theta^{\prime}}$ is nuclear.
\end{lemma}
\begin{proof}
By \cite[Proposition 2.3.11]{Pietsch}, it is enough to show that $\sum_{n\ge1}a_{n}(\iota:F_{\theta}\to F_{\theta^{\prime}})$ converges. 
For $m\in\N$, Lemma \ref{lem:681} (iv) and the inequality $\rank E_m\le N^m$ imply that $a_{N^m+1}(\iota:F_{\theta}\to F_{\theta^{\prime}})\le3(\theta/\theta^{\prime})^m$. 
Thus, we have
\begin{align*}
\sum_{n>N}a_{n}(\iota:F_{\theta}\to F_{\theta^{\prime}})
\le&\sum_{m\ge1}\,\sum_{N^m<n\le N^{m+1}}a_{N^m+1}(\iota:F_{\theta}\to F_{\theta^{\prime}})\\
\le&3(N-1)\sum_{m\ge1}\left(\frac{N\theta}{\theta^{\prime}}\right)^m<\infty,
\end{align*}
as desired.
\end{proof}

\begin{proof}[Proof of Theorem \ref{thm:988}]
For $\theta\in(0,1)$, $V_{\theta}$ denotes the completion of $V$ by the norm $\|\cdot\|_{\theta}$. 
If $0<\theta<\theta^{\prime}<1$, then 
\begin{equation}
\label{eq:143}
F_{\theta}\subset V_{\theta^{\prime}}\subset F_{\theta^{\prime}}.
\end{equation}
Indeed, $V_{\theta^{\prime}}\subset F_{\theta^{\prime}}$ is obvious and $F_{\theta}\subset V_{\theta^{\prime}}$ follows from Lemma \ref{lem:681} (iv).

Let $p$ be a continuous seminorm on $V$ and $\widehat{V}_{p}$ the completion of $V/\ker p$ by $p$. 
There exist $\theta^{\prime\prime}\in(0,1)$ and $C>0$ such that $p(\cdot)\le C\|\cdot\|_{\theta^{\prime\prime}}$. 
We set $\theta^{\prime}=\theta^{\prime\prime}/2,\,\theta=\theta^{\prime}/(N+1)$. 
By (\ref{eq:143}), the natural embedding $\iota:V_{\theta}\to\widehat{V}_{p}$ can be decomposed as follows:
\[
\iota:V_{\theta}\xrightarrow{\iota_1}(F_{\theta},\|\cdot\|_{\theta})\xrightarrow{\iota_2}(F_{\theta^{\prime}},\|\cdot\|_{\theta^{\prime}})\xrightarrow{\iota_3}(F_{\theta^{\prime}},\|\cdot\|_{\theta^{\prime\prime}})\xrightarrow{\iota_4}\widehat{V}_{p}.
\]
Here, all $\iota_1,\iota_2,\iota_3,\iota_4$ are the natural embeddings. 
Note that $\iota_2$ is nuclear from Lemma \ref{lem:343} and $\iota_1,\iota_3,\iota_4$ are continuous. 
Thus, $\iota:V_{\theta}\to\widehat{V}_{p}$ is nuclear from \cite[Proposition III.47.1]{Traves}, and hence, $V$ is a nuclear space.
\end{proof}

Recall from Section \ref{sec:introduction} that, for $q\in\N,\ \phi:\Sigma_{{\bf A}}^+\to\C$ and $\omega\in\Sigma_{{\bf A}}^+$, we write $S_q\phi(\omega)=\sum_{k=0}^{q-1}\phi(\sigma_{{\bf A}}^k\omega)$.

To prove Theorem \ref{thm:311}, we need some lemmas.

\begin{lemma}
\label{lem:711}
Let $m\in\N$. 
Let ${\bf A}$-admissible words $w,v$ satisfy $|w|\ge m,\,w_{0}=v_{0}$ and ${\bf A}(w_{|w|-1}w_{0})={\bf A}(v_{|v|-1}v_{0})=1$. 
Then, for $\varphi\in L_m$, we have
\[
S_{2|w|+|v|}\varphi((wwv)^*)=S_{|w|}\varphi(w^*)+S_{|w|+|v|}\varphi((wv)^*).
\]
\end{lemma}
\begin{proof}
By $S_{p+q}\varphi(\omega)=S_p\varphi(\omega)+S_q\varphi(\sigma_{{\bf A}}^p\omega)$, we have
\[
S_{2|w|+|v|}\varphi((wwv)^*)=S_{|w|}\varphi((wwv)^*)+S_{|w|}\varphi(wv(wwv)^*)+S_{|v|}\varphi(v(wwv)^*).
\]
Notice that $S_{|w|}\varphi((wwv)^*)=\sum_{k=0}^{|w|-1}\varphi(w_{k}\cdots w_{|w|-1}wv(wwv)^*)$. 
Since $\varphi\in L_m$ and $(w_{k}\cdots w_{|w|-1}wv(wwv)^*)|m=(w_{k}\cdots w_{|w|-1}w^*)|m$ for $0\le k\le|w|-1$, we have $\varphi(w_{k}\cdots w_{|w|-1}wv(wwv)^*)=\varphi(w_{k}\cdots w_{|w|-1}w^*)$ for $0\le k\le|w|-1$. Hence, $S_{|w|}\varphi((wwv)^*)=S_{|w|}\varphi(w^*)$. 
Similarly, $S_{|w|}\varphi(wv(wwv)^*)=S_{|w|}\varphi((wv)^*)$ and $S_{|v|}\varphi(v(wwv)^*)=S_{|v|}\varphi(\sigma_{{\bf A}}^{|w|}(wv)^*)$.
Therefore, we have 
\begin{align*}
S_{2|w|+|v|}\varphi((wwv)^*)=&S_{|w|}\varphi(w^*)+S_{|w|}\varphi((wv)^*)+S_{|v|}\varphi(\sigma_{{\bf A}}^{|w|}(wv)^*)\\
=&S_{|w|}\varphi(w^*)+S_{|w|+|v|}\varphi((wv)^*),
\end{align*}
as desired.
\end{proof}

Here is a key lemma.

\begin{lemma}
\label{lem:414}
There exist ${\bf A}$-admissible words $\tilde{w},v$ such that the following two assertions hold:
\begin{itemize}
\item[(i)]$\tilde{w}_0=v_0,\ {\bf A}(\tilde{w}_{|\tilde{w}|-1}\tilde{w}_0)={\bf A}(v_{|v|-1}v_0)=1$ and $|\tilde{w}|\ge|v|$.
\item[(ii)]$v$ is self-avoiding and $v_{|v|-1}\notin\{\tilde{w}_k:0\le k\le|\tilde{w}|-1\}$.
\end{itemize}
\end{lemma}
\begin{proof}
From Lemma \ref{lem:203}, there exists $i\in\{1,\dots,N\}$ such that the $i$-th column has more than two entries which are equal to one.
Take an ${\bf A}$-admissible word $\bar{w}$ so that $\bar{w}_0=\bar{w}_{|\bar{w}|-1}=i$ and
\begin{equation}
\label{eq:039}
|\bar{w}|=\min\{|w|:w\ \mbox{is}\ {\bf A}\mbox{-admissible},\ w_0=w_{|w|-1}=i\}.
\end{equation}

Let $j_1=\bar{w}_{|\bar{w}|-2}$. 
There exists $j_2\in\{1,\dots,N\}$ such that $j_2\neq j_1$ and ${\bf A}(j_2i)=1$. 
We prove the following assertion:
\begin{equation}
\label{eq:167}
j_2\notin\{\bar{w}_k:0\le k\le|\bar{w}|-2\}.
\end{equation}
First, $j_2\neq\bar{w}_{|\bar{w}|-2}$ since $\bar{w}_{|\bar{w}|-2}=j_1$.
Next, we show that $j_2\neq\bar{w}_0$. 
Recall $\bar{w}_0=i$. 
We assume $j_2=i$. 
Then, ${\bf A}(ii)=1$, and hence, from (\ref{eq:039}), $\bar{w}=ii$. 
Therefore, we have the contradiction $j_1=i=j_2$.
Finally, we show that $j_2\neq\bar{w}_k$ for any $k\in\{1,\dots,|\bar{w}|-3\}$.
We assume $ j_2=\bar{w}_k$ for some $k\in\{1,\dots,|\bar{w}|-3\}$.
Then, the word $w=\bar{w}_0\bar{w}_{1}\cdots\bar{w}_k\,i$ is ${\bf A}$-admissible and satisfies $w_0=w_{|w|-1}=i$. 
Moreover, $|w|=k+2<|\bar{w}|$. 
Thus, by (\ref{eq:039}), we have a contradiction, and (\ref{eq:167}) follows.

We prove the lemma.
Take an ${\bf A}$-admissible word $v$ so that $v$ is self-avoiding and $v_0=i,\ v_{|v|-1}=j_2$.
Moreover, take $N\in\N$ so that $N(|\bar{w}|-1)\ge|v|$ and set
\[
\tilde{w}=\overbrace{(\bar{w}_0\cdots\bar{w}_{|\bar{w}|-2})(\bar{w}_0\cdots\bar{w}_{|\bar{w}|-2})\cdots(\bar{w}_0\cdots\bar{w}_{|\bar{w}|-2})}^{\text{$N$ times}}.
\]
Then, these $\tilde{w},v$ satisfy the desired properties.
\end{proof}

Let $\tilde{w},v$ be as in Lemma \ref{lem:414}.
For $m\in\N$, we write
\[
w^{(m)}=\overbrace{\tilde{w}\,\tilde{w}\cdots\tilde{w}}^{\text{$m$ times}}.
\]

\begin{lemma}
\label{lem:611}
Let $m\in\N$ and let $\phi:\Sigma_{{\bf A}}^+\to\C$ be cohomologous with an $m$-locally constant function. 
Then, we have
\begin{equation}
\label{eq:446}
S_{2|w^{(m)}|+|v|}\phi((w^{(m)}w^{(m)}v)^*)=S_{|w^{(m)}|}\phi((w^{(m)})^*)+S_{|w^{(m)}|+|v|}\phi((w^{(m)}v)^*).
\end{equation}
\end{lemma}
\begin{proof}
We write $w=w^{(m)}$ for the sake of simplicity.
Let $\phi$ be cohomologous with $\varphi\in L_m$. 
Lemma \ref{lem:414} (ii) implies that the periodic point $(wwv)^*$ has period $2|w|+|v|$. 
Thus, we have $S_{2|w|+|v|}\phi((wwv)^*)=S_{2|w|+|v|}\varphi((wwv)^*)$. 
Similarly, we have $S_{|w|}\phi(w^*)=S_{|w|}\varphi(w^*)$ and $S_{|w|+|v|}\phi((wv)^*)=S_{|w|+|v|}\varphi((wv)^*)$. 
Thus, we obtain (\ref{eq:446}) by Lemma \ref{lem:711}.
\end{proof}

We are ready to prove Theorem \ref{thm:311}.

\begin{proof}[Proof of Theorem \ref{thm:311}]
For $m\in\N$, we set
\[
V_m=\{\phi\in V\>|\>\phi\ \mbox{does not satisfy (\ref{eq:446})}\}.
\]
Lemma \ref{lem:611} implies that $R\supset\bigcap_{m\in\N}V_m$. 
Thus, it is enough to show that, for $m\in\N$, $V_m$ is an open and dense subset of $V$. 

We write $w=w^{(m)}$. 
Then, the two maps $\phi\mapsto S_{2|w|+|v|}\phi((wwv)^*)$ and $\phi\mapsto S_{|w|}\phi(w^*)+S_{|w|+|v|}\phi((wv)^*)$
from $V$ to $\C$ are continuous, and hence, $V_m$ is open.

We prove the denseness. 
Let $\phi\in V\setminus V_m$. 
We write $\epsilon_n=(1/n)1_{[wwv]}$ for $n\in\N$. 
Then, for $\theta\in(0,1)$, we have $\|\epsilon_n\|_{\theta}\le(1/n)(1+2/\theta^{2|w|+|v|-1})$, and hence, $\|(\phi+\epsilon_n)-\phi\|_{\theta}\to0$ as $n\to\infty$.
Thus, $\phi+\epsilon_n\to\phi\>\text{in}\>V$ as $n\to\infty$. 
Therefore, it is enough to prove
\begin{equation}
\label{eq:415}
\phi+\epsilon_n\in V_m,\quad n\in\N.
\end{equation} 
We easily have $S_{|w|}\epsilon_{n}(w^*)=0$.
We show that $S_{|w|+|v|}\epsilon_{n}((wv)^*)=0$. 
It is enough to show that 
\begin{equation}
\label{eq:331}
\sigma_{{\bf A}}^k(wv)^*|(2|w|-1)\neq wwv,\quad0\le k\le|w|+|v|-1. 
\end{equation}
Let $0\le k\le|w|+|v|-1$. 
Then, $(\sigma_{{\bf A}}^{k}(wv)^*)_{|w|+|v|-1-k}=(wv)^*_{|w|+|v|-1}=v_{|v|-1}$.
On the other hand, $|w|+|v|-1-k\le2|w|-1$
from $|w|\ge|\tilde{w}|\ge|v|$, and hence, $(wwv)_{|w|+|v|-1-k}\neq v_{|v|-1}$.
Thus, (\ref{eq:331}) holds.

We prove (\ref{eq:415}). 
We have $S_{2|w|+|v|}(\phi+\epsilon_n)((wwv)^*)\neq S_{2|w|+|v|}\phi((wwv)^*)$ since $S_{2|w|+|v|}\epsilon_n((wwv)^*)\ge1/n>0$.
Lemma \ref{lem:611} and $S_{|w|}\epsilon_{n}(w^*)=S_{|w|+|v|}\epsilon_{n}((wv)^*)=0$ imply $S_{2|w|+|v|}\phi((wwv)^*)=S_{|w|}(\phi+\epsilon_n)(w^*)+S_{|w|+|v|}(\phi+\epsilon_n)((wv)^*)$.
Thus, $\phi+\epsilon_n$ does not satisfy (\ref{eq:446}), and we obtain (\ref{eq:415}).
\end{proof}


\section{Asymptotic behavior of the eigenvalues}
\label{appendix:eigenvalues}

In this appendix, we give the following asymptotic behavior of eigenvalues of transfer operators:

\begin{theorem}
\label{thm:885}
Let $\{\theta_m\}$ satisfy (\ref{cond:227}) and assume that $\theta_m>0$ for any $m\in\N$.
Let $C_2$ be as in Lemma \ref{lem:216}. 
Let $\alpha>0$ and $R>e^{h_{\mathrm{top}}(\sigma_{{\bf A}})}$. 
Then, there exist $C_5>0$ and $M\in\N$, depending only on $b_1,b_2,\alpha$ and $R$, such that the following inequality holds for $m\ge M$ and $g\in V$ satisfying (\ref{condi:029}):
\[
\#\{n\in\N:|\lambda_n(g)|>(C_2+1)\theta_{m}\}\le C_5\theta_{m}^{-\alpha}R^{m-1}.
\]
\end{theorem}
\begin{proof}
We write $\mcB=\mcB(\{\theta_m\})$.
Take $M\in\N$ so that $\theta_{m}\le1$ and $\rank E_{m-1}\le R^{m-1}$ for $m\ge M$. 
For $s>0$, we denote by $N_m(s)$ the number of eigenvalues $\lambda$ of $\mathscr{L}_{g}:\mcB\to\mcB$ with $|\lambda|>s$, where each $\lambda$ is counted according to its multiplicity. 
Then, by Theorem \ref{thm:919} (ii), $\#\{n\in\N:|\lambda_n(f)|>(C_2+1)\theta_{m}\}=N_m((C_2+1)\theta_{m})$. 
We see by (\ref{ineq:094}) that $(C_2+1)\theta_{m}-\|\mathscr{L}_{g}-K_{g,m-1}\|_{\mcB\to\mcB}\ge\theta_m$. 
Thus, by \cite[Corollary 4.3]{Demuth et al JFA 2015}, there exists $C(\alpha)>0$ such that the following inequality holds:
\[
N_m((C_2+1)\theta_m)\le
C(\alpha)(C_2+1)\theta_{m}^{-\alpha}\sum_{k=1}^{\infty}a_{k}(K_{g,m-1})^{\alpha}.
\]
Since $a_k(K_{g,m-1})=0$ for $k>\rank K_{g,m-1}$ and $\rank K_{g,m-1}\le\rank E_{m-1}$, we see that $\sum_{k=1}^{\infty}a_{k}(K_{g,m-1})^{\alpha}\le R^{m-1}\|K_{g,m-1}\|_{\mcB\to\mcB}^{\alpha}$.
From (\ref{ineq:278}) and Lemma \ref{lem:046} (iii), $\|K_{g,m-1}\|_{\mcB\to\mcB}\le4C_2$. 
Thus, the assertion holds for $C_5=(C_2+1)(4C_2)^{\alpha}C(\alpha)$.
\end{proof}


\section*{Acknowledgments}
The author would like to thank Akihiko Inoue (Hiroshima University) and Yushi Nakano (Tokai University) for helpful advices and comments. 
The author is also grateful to the anonymous reviewer for his/her careful reading and many insightful comments and suggestions.
The construction of the Banach space $\mcB$ on which the transfer operator is compact is based on the reviewer's idea.
The author is supported by FY2019 Hiroshima University Grant-in-Aid for Exploratory Research (The researcher support of young Scientists).



\begin{thebibliography}{99}

\bibitem{Baladi Contemp.Math 2008} 
     \newblock  V. Baladi and M. Tsujii, 
     \newblock Dynamical determinants and spectrum for hyperbolic diffeomorphisms,
     \newblock in \emph{Geometric and Probabilistic Structures in Dynamics} (eds. K. Burns, D. Dolgopyat and Ya. Pesin), \emph{Contemp. Math.}, \textbf{469} (Amer. Math. Soc.), (2008), 29--68.

\bibitem{Boas} 
     \newblock  R. Boas, 
     \newblock \emph{Entire Functions},
     \newblock Academic Press Inc., New York, 1954.

\bibitem{Demuth et al JFA 2015} 
     \newblock  M. Demuth, F. Hanauska, M. Hansmann and G. Katriel, 
     \newblock Estimating the number of eigenvalues of linear operators on Banach spaces,
     \newblock \emph{J. Funct. Anal.}, \textbf{268} (2015), 1032--1052.

\bibitem{Fried 1986} 
     \newblock  D. Fried, 
     \newblock The zeta functions of Ruelle and Selberg. I,
     \newblock \emph{Ann. Sci. \'{E}cole. Norm. Sup. (4)}, \textbf{19} (1986), 491--517.

\bibitem{Haydn ETDS 1990} 
     \newblock  N. T. A. Haydn, 
     \newblock Meromorphic extension of the zeta function for Axiom A flows,
     \newblock \emph{Ergodic Theory Dynam. Systems}, \textbf{10} (1990), 347--360.

\bibitem{Jezequel 2018} 
     \newblock  M. J\'{e}z\'{e}quel, 
     \newblock Local and global trace formulae for smooth hyperbolic diffeomorphisms,
     \newblock \emph{J. Spectr. Theory}, \textbf{10} (2020), 185-249.

\bibitem{Konig} 
     \newblock H. K\"{o}nig, 
     \newblock \emph{Eigenvalue Distribution of Compact Operators},
     \newblock Birkh\"{a}user Verlag, Basel, 1986.

\bibitem{Parry-Pollicott} 
\newblock W. Parry and M. Pollicott, 
\newblock \emph{Zeta Functions and the Periodic Orbit Structure of Hyperbolic
Dynamics},
\newblock Ast\'{e}risque (No. 187--188), Paris, 1990.

\bibitem{Pietsch} 
\newblock A. Pietsch, 
\newblock \emph{Eigenvalues and $s$-Numbers},
\newblock Cambridge University Press, Cambridge, 1987.

\bibitem{Pollicott Invent 1986} 
\newblock  M. Pollicott, 
\newblock  Meromorphic extensions of generalised zeta functions, 
\newblock \emph{Invent. Math.}, \textbf{85} (1986), 147--164.

\bibitem{Quas ETDS 2012} 
     \newblock  A. Quas and J. Siefken, 
     \newblock Ergodic optimization of super-continuous functions on shift spaces,
     \newblock \emph{Ergodic Theory Dynam. Systems}, \textbf{32} (2012), 2071--2082.

\bibitem{Ruelle Invent 76} 
\newblock  D. Ruelle, 
\newblock  Zeta-functions for expanding maps and Anosov flows, 
\newblock \emph{Invent. Math.}, \textbf{34} (1976), 231--242.

\bibitem{Ruelle IHES 90} 
\newblock  D. Ruelle, 
\newblock  An extension of the theory of Fredholm determinants, 
\newblock \emph{Inst. Hautes \'{E}tudes Sci. Publ. Math.}, \textbf{72} (1990), 175--193.

\bibitem{Traves} 
\newblock F. Tr\`{e}ves, 
\newblock \emph{Topological Vector Spaces, Distributions and Kernels},
\newblock  Academic Press, New York-London, 1967.
\end{thebibliography}
\end{document}